\documentclass[a4paper,12pt]{article}
\usepackage{amsmath, amssymb, amsthm}
\usepackage{amscd}
\usepackage[all]{xy}

\newtheorem{thm}{Theorem}
\newtheorem*{thmA}{Theorem A}
\newtheorem*{thmB}{Theorem B}
\newtheorem*{thmF}{Theorem F}
\newtheorem{lemma}[thm]{Lemma}
\newtheorem{sublem}[thm]{Sublemma}

\newtheorem{defn}[thm]{Definition}

\newtheorem{prop}[thm]{Proposition}

\newtheorem{cor}[thm]{Corollary}
\newtheorem*{corC}{Corollary C}
\newtheorem*{corD}{Corollary D}
\newtheorem*{corE}{Corollary E}

\theoremstyle{remark}

\newtheorem{prob}{Problem}

\newtheorem{ex}[thm]{Example}

\DeclareMathOperator{\M}{\mathcal{M}}

\DeclareMathOperator{\Diff}{Diff}
\DeclareMathOperator{\PH}{\mathcal{PH}}
\DeclareMathOperator{\AM}{\mathcal{AM}}

\DeclareMathOperator{\Graph}{graph}

\DeclareMathOperator{\supp}{supp}
\DeclareMathOperator{\ess}{ess}
\DeclareMathOperator{\vol}{vol}
\DeclareMathOperator{\Lip}{Lip}
\DeclareMathOperator{\Jac}{Jac}
\DeclareMathOperator{\B}{\mathcal{B}}
\DeclareMathOperator{\T}{\mathcal{T}_{\epsilon}}
\DeclareMathOperator{\umu}{\underline{\mu}}
\DeclareMathOperator{\unu}{\underline{\nu}}
\DeclareMathOperator{\A}{\mathbf{A}}
\DeclareMathOperator{\F}{\mathcal{F}}
\DeclareMathOperator{\K}{\mathbf{K}}
\DeclareMathOperator{\Dl}{Diff_{loc}^2}

\DeclareMathOperator{\uiota}{\underline{\iota}}

\DeclareMathOperator{\mue}{\mu_{\epsilon}}
\DeclareMathOperator{\nue}{\nu_{\epsilon}}
\DeclareMathOperator{\ff}{\mathbf{f}}
\DeclareMathOperator{\tmu}{\tilde{\mu}}

\title{Robust Ergodic Properties in Partially Hyperbolic Dynamics}
\author{Martin Andersson\footnote{This work was supported by CNPq (Brazil).}}

\begin{document}

\maketitle

\begin{abstract}

We study ergodic properties of partially hyperbolic systems  whose central direction is mostly contracting. Earlier work of Bonatti, Viana \cite{BV} about existence and finitude of physical measures is extended to the case of local diffeomorphisms. Moreover, we prove that such systems constitute a $C^2$-open set in which statistical stability is a dense property. In contrast, \emph{all} mostly contracting systems are shown to be stable under small random perturbations.
\end{abstract}

\section{Introduction}

A sound approach to understanding smooth dynamical systems consists of giving a statistical description of most orbits. It is sensible due to the extreme complexity of the orbit structures, so frequently encountered in dynamical systems with some expanding behaviour. In practice this often boils down to finding
out whether a given system $f$ has a \emph{physical measure}, i.e. a probability measure $\mu$ for which the \emph{basin}
\begin{equation}\label{basin}
\B(\mu):= \{ x\in M: \frac{1}{n} \sum_{k=0}^{n-1} \delta_{f^k(x)} \overset{\text{weakly}}{\rightarrow} \mu \}
\end{equation}
has positive Lebesgue measure. Successful work on Axiom A diffeomorphisms \cite{Ru, Si, Y} has lead dynamiscists to believe that many dynamical systems can be satisfactorily described on a statistical basis --- a view taken by Palis in his well-known conjecture on the denseness of finitude of attractors \cite{P}. A description of a dynamical system in terms of physical measures can be considered rather complete if it encapsulates topics like

\begin{description}
\item[Existence] There are physical measures for the system.
\item[Finitude] The number of physical measures is finite.
\item[No holes] Lebesgue almost every point in the manifold $M$ belongs to the basin of some physical measure.
\item[Statistical stability] All physical measures persist under small perturbations.
\item[Stochastic stability] Physical measures describe random orbits of the system under small noise.
\end{description}

Since the seventies, physical measures have been proved to exist in much greater generality than Axiom A diffeomorphisms, including some partially hyperbolic systems \cite{BV, ABV, T}, the H\'enon family \cite{BY}, and others. 

In the present work, we study an open set of $C^2$ partially hyperbolic local diffeomorphisms 
$f:M \rightarrow M$ on compact Riemannian manifolds, mostly contracting along the central direction. Such systems provide a non-invertible generalization of mostly contracting diffeomorphisns, first studied by Bonatti, Viana \cite{BV}, and later by Castro \cite{C} and Dolgopyat \cite{D}; however this time the focus is on statistical stability. Particularily under the possibility of coexistence of several physical measures on the same attractor.

A conceivable obstacle to statistical stability is the seemingly pathological phenomenon, present in a fascinating example due to Kan \cite{K}, exhibiting two physical measures supported in the same transitive piece of the dynamics. It seems likely that this phenomenon can be destroyed by small perturbations of the system, thus leading to a bifurcation in the set of physical measures. Kan's example falls into a class of systems which we nowadays call partially hyperbolic with mostly contracting central direction. It is known from the work of \cite{BV} that if the unstable foliation is minimal for a mostly contracting diffeomorphism, then there is only one physical measure. To what extent this  occurs is not known, although some research has been made on the subject \cite{BDU, PuSa},  suggesting it to be a common feature. 

Nevertheless, the present work introduces a new set of techniques to deal with statistical (and stochastic) stability of mostly contracting systems, independently of whether they exhibit Kan's phenomenon or not. We prove:
\begin{itemize}
 \item Mostly contracting contracting local diffeomorphisms have a finite number of physical 
measures and satisfy the no holes property.
\item Having mostly contracting central direction is a robust property.
\item The number of physical measures vary semi-continuously with the dynamics.
\item Sytems that do not alter the number of physical measures under small perturbations are statistically stable.
\item These make up an open and dense subset of all mostly contracting systems.
\item In particular, all systems with a unique physical measure are statistically stable.
\item Among mostly contracting conservative diffeomorphisms, every ergodic system is necessarily stably ergodic.
\item All mostly contracting systems are stochastically stable. 
\end{itemize}

A key feature of the arguments used is that they apply to non-invertible maps just as well as diffeomorphisms, provided that there are no critical points. This is done by replacing the traditional Gibbs-$u$ states \cite{PeSi} with a multi-dimensional analogue of Tsujii's admissible measures \cite{T}.
The current approach is even more advantageous in the non-invertible case, where uniqueness of the physical measure is harder to obtain due to the lack of unstable foliation.

\subsection*{Acknowledgements}

This being my first independent work in mathematical research, I would like to seize the oportunity to thank all those people that have tought me mathematics; particularily those who have put trust into my academic progress. I am refering here to Stefano Luzatto, who introduced me to dynamical systems, and Marcelo Viana, for accepting me as his student at IMPA where I have learnt most of what I know in the field. I also thank Fl\'avio Abdenur for acting as an encouraging force and for being such a fierce promoter of semi-continuity arguments. My fellow student Yang Jiagang owes a great thank for pointing out Corollary D and its proof, as do all other students at IMPA, with whom I have exchanged ideas on a daily basis. My last mention goes to Augusta for providing such a divine Bob\'o de Camar\~ao --- it has certainly had a good effect on my work!

\section{Some preliminary notions and description of results}

Let $M$ a smooth compact Riemannian manifold. To avoid trivial statements, we will suppose the dimension to be at least two. Denote by $\Dl(M)$ the space of $C^2$ local diffeomorphisms on $M$, i.e. $C^2$ maps whose derivative is of full rank at every point. It is an open subspace of $C^2(M, M)$ and, in particular, contains all diffeomorphisms. Elements of $\Dl(M)$ will be  referred to as \emph{systems}, or simply \emph{maps}.

We deviate slightly from standard terminology and say that $\Lambda$ is an attractor for the system $f$ if $\Lambda$ is a compact $f$-invariant set and there exists an open neighbourhood $U$ of $\Lambda$, called a \emph{trapping region}, such that 
\[f(\overline{U}) \subset U \text{ and } \Lambda = \bigcap_{n \geq 0} f^n(U). \]
In other words, there is no requirement of transitivity and, in particular, $M$ itself is always an attractor with trapping region $M$.

\subsection{Partial Hyperbolicity}

Several notions of partial hyperbolicity may currently be found in the literature, of which the most widely known requires a decomposition of the tangent bundle into three complementary subbundles (see \cite{AV} for discussion). The type cinsidered in this work  requires only two complementary subbundles, one of which is uniformely expanded under the action of the system and dominating the other. It is usually referred to as partial hyperbolicity of type $E^u \oplus E^{cs}$.

Thus an attractor $\Lambda$ is \emph{partially hyperbolic} under $f$ if there exists a splitting 
$T_{\Lambda}M = E^c \oplus E^u$ into non-trivial subspaces,  a constant $0<\tau<1$, and an integer $n_0$ such that 
\begin{eqnarray}\label{expansion}
\|(Df_{\vert_{E_x^u}}^n)^{-1}\|&\leq& \tau^{n-n_0} \\
\label{domination}
\|Df_{\vert_{E_x^c}}^n\| \|(Df_{\vert_{E_x^u}}^n)^{-1}\| &\leq& \tau^{n-n_0}
\end{eqnarray}
both hold for every $x\in \Lambda$ and every $n \geq 0$.

The subspace $E_x^c$ above is necessarily unique, and varies continuously with $x$. On the other hand, $E_x^u$ is not. In fact, when $f$ is non-invertible, there is typically no invariant unstable direction at all. Still, we can always define a strictly invariant conefield 
\[
 S_x^u = \{v^c \oplus v^u \in E_x^c \oplus E_x^u: v^u \geq \alpha v^c\}
\]
 for some $\alpha>0$. Strict invariance here means that $Df_x S_x^u$ is contained in the 
interior of $S_{f(x)}^u$ for every $x\in U$.  The subspace $E_x^c$ is characterised by those vectors $v\in T_x M$ 
such that $Df_x^n v\notin S_{f^n(x)}^u$ for every $n \geq0$. There is no harm in supposing that $E^u$ is smooth, say $C^{\infty}$. The lack of invariance of $E^u$ is reflected in the following observation: Let $\ldots x_{-2}, x_{-1}, x_0, \ \ldots y_{-2}, y_{-1}, y_0$ be two different pre-orbits of a point $x_0=y_0$. Then $\bigcap_{n\geq0} Df^nS_{x_{-n}}^u$ is not necessarily the same as $\bigcap_{n\geq0} Df^nS_{y_{-n}}^u$. 

Upon possibly replacing $U$ by a subset, and slightly altering the constants $n_0, \tau$, we may suppose that the splitting and unstable cone field extend to the whole of $U$, and (\ref{expansion}), (\ref{domination}) hold for every $x \in U$.

We denote by $\PH(U, S^u)$ those $f \in \Dl(U)$ that leave $U$ and $S^u$ strictly invariant and admit a partially hyperbolic splitting satisfying (\ref{expansion}), (\ref{domination}) for some $\tau<1$. It is an open subset of $\Dl(M)$. 

We call $E^c$ the \emph{central direction} of $f$, and use the notation 
\[D^cf := Df_{\vert E^c}\]
in all that follows. The letters $c$ and $u$ will also denote the dimensions of $E^c$ and $E^u$ --- the central and unstable dimensions.

\subsection{Mostly contracting central direction}\label{mostcont}

The maximum central Lyapunov exponent is the map 
\begin{eqnarray*}
\lambda_+^c: \PH(U, S^u)\times U &\rightarrow& \mathbb{R} \\ 
(f, x) &\mapsto& \limsup_{n \rightarrow \infty} \frac{1}{n} \log \|D^c f^n(x)\|. 
\end{eqnarray*}

We rephrase the definition of mostly contracting diffeomorphisms used in \cite{BV} suitably into our context.

\begin{defn}\label{mostly}
A system $f\in \PH(U, S^u)$ is mostly contracting along the central direction if, given any disc $D \subset U$ (at least $C^{1+ \Lip}$) tangent to $S^u$, there exists a subset $A \subset D$ of positive Lebesgue measure such that $\lambda_+^c(f, x) < 0$ for every $x \in A$. 
\end{defn}

After characterising this definition in Section \ref{character}, it will become clear that it coincides with that of \cite{BV} in the case of diffeomorphisms.
The space of mostly contracting systems in $\PH(U, S^u)$ will be denoted by $MC(U, S^u)$. We shall be irresponsible and omit explicit mentioning of the trapping region and unstable conefield. Thus when saying that $f$ is partially hyperbolic ($f\in \PH$), it is understood that there exists some  trapping region $U$ and a dominated splitting $T_U M=E^u\oplus E^c$ with associated invariant conefield $S^u$, constants $\tau, n_0$ satisfying (\ref{expansion}) and (\ref{domination}) for every $x \in U$ and $n \geq 0$. Similarly for $MC$. All objects except $E^c$ can be applied on maps in some $C^2$ neighbourhood of $f \in \PH(U, S^u)$ to yield partial hyperbolicity. The central distribution $E^c$ varies with the map, although in a continuous fashion.

The mostly contracting condition was created in \cite{BV} to prove existence, finitude and the no holes property of physical measures for partially hyperbolic diffeomorphisms. We are going to develop techniques that allow for a generalisation of their result into a non-invertible context.
 
\begin{thmA}\label{thmA}
 Every $f$ in $MC$ possess a finite number of physical measures and the union of their basins 
of attraction cover Lebesgue almost every point of $U$.
\end{thmA}

\subsection{Robustness and Statistical stability}

The main theorem in this paper addresses robustness properties of maps in $MC$.
It is not clear from the definition whether the mostly contracting condition is open or not. Neither does Theorem A (nor its predecessor Theorem A in \cite{BV}) give any hint as to what might happen with the physical measures under small perturbations of the map in question. In his article \cite{D}, Dolgopyat addresses these kind of questions for some mostly contracting systems on three dimensional manifolds, satisfying some additional properties which in particular imply uniqueness of the physical measure. He achieves statistical stability and strong statistical properties such as exponential decay of correlations. The intention of this work is rather different, as we will not bother about the number of physical measures. Nor do we study any strong statistical properties, but will only be concerned with looking at how the physical measures depend on the system.

\begin{defn}
Let $f\in \Dl(M)$ be a system having a finite number of physical measures $\mu_1, \ldots, \mu_N$ in some trapping region $U$. 
We say that $f$ (strictly speaking  the pair $(f, U)$) is statistically stable if there exists a neighbourhood $\mathcal{U}$ of $f$, and weakly continuous functions 
\[\Phi_1, \ldots, \Phi_N: \mathcal{U} \rightarrow \M(M)\] such that, given any $g\in \mathcal{U}$, the physical measures of $g$, supported in $U$, coincide precisely with $\Phi_1(g), \ldots, \Phi_N(g)$.

Similarly, given any subset $\mathcal{C}\subset\Dl(U)$, we define statistical stability under perturbations within $\mathcal{C}$ by requiring that the functions $\Phi_1, \ldots, \Phi_N$ be defined on $\mathcal{C}$ only.
\end{defn}

\begin{thmB}\label{thm2}
\

\begin{enumerate}

\item
$MC$ is open in the $C^2$ topology.

\item\label{item2}
The number of physical measures supported in $U$ is an upper
semi-continuous function $MC\rightarrow \mathbb{N}$.

\item\label{itemC}
Let $\mathcal{C}$ be any subset of $MC$ such that the number of physical measures
supported in $U$ is constant for maps in $\mathcal{C}$. Then maps in $\mathcal{C}$ are statistically stable under
perturbations within $\mathcal{C}$.
\end{enumerate}
\end{thmB}

By our choice of definition, statistical stability does not make sense if the number of physical measures changes abruptly. Theorem B states that whenever statistical stability makes sense, it holds. In other words, a drop in the number of physical measures is the \emph{only} obstacle to statistical stability among mostly contracting systems, so Theorem B is the strongest possible result of its kind. Let us take a look at some of its consequences, the first of which is immediate. 

\begin{corC}
 Maps in $MC$ having precisely one physical 
measure form an open set, and are therefore statistically stable.
\end{corC}

As we shall see in Section \ref{staberg}, Corollary C takes a particularily nice form when applied to conservative systems. But first, let us see how simple semi-continuity arguments may be applied to prove great abundance of statistical stability among mostly contracting systems.

\begin{corD}
 Statistical stability is an open and dense property in $MC$.
\end{corD}
\begin{proof}
 For $n\geq 1$, let $\mathcal{S}_n$ be the set of maps in $MC$ 
having at most $n$ physical measures. 
By semi-continuity, each $\mathcal{S}_n$ is open. We define $\mathcal{O}_1 = \mathcal{S}_1$ and $\mathcal{O}_{n+1} = \mathcal{S}_{n+1} \setminus \overline{\mathcal{S}}_n$ for every $n \geq 1$. Then each $\mathcal{O}_n$ is an open set on which the number of physical measures is precisely $n$. Hence every map in $\mathcal{O} = \bigcup_{n\geq1} \mathcal{O}_n$ is statistically stable and, by construction, $\mathcal{O}$ is dense.
\end{proof}

\subsection{Stable ergodicity}\label{staberg}

There is a noteworthy application of Theorem B to the theory of stable ergodicity. We say that a diffeomorphism is conservative if it preserves Lebesgue measure on $M$, and we denote the space of all conservative maps by $\Diff_m^2(M)$. 

\begin{defn}
Let $f \in \Diff_m^2(M)$. We say that $f$ is \emph{stably ergodic} if there exists a $C^2$ neighbourhood $\mathcal{U}$ of $f$  such that Lebesgue measure is ergodic under every $g\in \mathcal{U} \cap \Diff_m^2(M)$.
\end{defn}

Partial hyperbolicity is believed to be a strong mechanism for stable ergodicity. See \cite{PuSh} for details.

\begin{corE}
Any ergodic diffeomorphism in $MC$ is automatically stably ergodic.
\end{corE}

This is not the first time stable ergodicity has been considered fror mostly contracting systems. In \cite{BDP}, the authors give a condition (Theorem 4) of stable ergodicity for mostly contracting systems. The point here is that nothing at all has to be said about the neighbours of $f$, but that ergodicity really is a robust (open) property in $MC \cap \Diff_m^2(M)$. Clearly the same can be said about local diffeomorphisms, although today's research interest in stable ergodicity does not reach outside the world of diffeomorphisms (as far as I know).

\subsection{Some related problems}

 Suppose that $A, B \subset M$ are two Borel subsets, 
each of positive Lebesgue measure $m$. To set some terminology, let us say that $A$ and $B$ \emph{emulsify} if  $ \supp (m_{\vert A}) \cap \supp (m_{\vert B})$
has non-empty interior. Kan's example \cite{K} shows that a mostly contracting system may possess two physical measures with emulsifying basins.

\begin{prob}
Are there robust examples (in $MC$ or elsewhere) of systems having physical measures with basins in emulsion?
\end{prob}

\begin{prob}
Do bifurcations (descontinuities in the number) of physical measures really take place for mostly contracting systems? In particular, the example of Kan as described in \cite{K} is an endomorphism on the cylinder $S^1\times [0,1]$. But it can easily be turned into a local diffeomorphism on the torus $\mathbb{T}^2$ by gluing two copies together. Is it then, one may ask, possible to perform a small $C^2$ perturbation in such a way that the resulting system has only one physical measure?
\end{prob}

Let $X$ be the family of all Borel subsets of $M$ up to equivalence of zero Lebesgue measure: $A \sim B$ iff $m(A \Delta B) = 0$. We endow $X$ with the metric $d$ of symmetric difference on $X$, i.e. $d(A, B) = m(A \Delta B)$ for $A, B \in X$.

\begin{prob}
 Suppose $f \in MC$ is statistically stable. 
Do the basins of its physical measures vary continuously on $f$ in the topology of symmetric difference?
\end{prob}

\subsection{Stochastic stability}

We give only a brief account of noise modelling and stochastic stability of dynamical systems, recomending \cite{Ki} for a more detailed exposition.

Let $f\in \Dl(M)$ and $\{\nu_{\epsilon}\}_{\epsilon>0}$ be a family of probability measures in $\Dl(M)$ supported in $C^2$-balls $B_{\epsilon}(f)$. We think of $f$ as being a model for a scientific phenomena and, for each $\epsilon$, $\nu_{\epsilon}$ is to be thought of as random noise corresponding to external effects not accounted for by the model. The number $\epsilon$ is the \emph{magnitude}, or \emph{level} of the noise.

The family $\{\nu_{\epsilon}\}_{\epsilon>0}$ gives rise to a family $\{\T\}_{\epsilon>0}$ of operators on $\M(M)$, given by
\[\T \mu = \int_{\Dl(M)} f_*\mu \ d\nu_{\epsilon}(f).\]

Since $\nu_{\epsilon}$ is contained in a $C^2$ ball of $f$, it follows that 
\begin{equation}\label{local}
 \supp \T \delta_x \in B_{\epsilon}(f(x)) \quad \forall x \in M.
\end{equation}
We refer to the property (\ref{local}) by saying that the perturbations are \emph{local}. In other words, the random image of any point $x$ is almost surely $\epsilon$ close to the deterministic image $f(x)$.

Another property imposed on the family $\{\nu_{\epsilon}\}_{\epsilon>0}$ so that it provides a realistic model of noise, is that it be \emph{absolutely continuous}:
\begin{equation}\label{abscont}
 \T \delta_x << \text{Leb} \quad \forall x \in M.
\end{equation}

Being $\T$ linear continuous, the Krylov-Bogolyubov argument proves the existence of invariant distributions $\mu_{\epsilon} = \T \mu_{\epsilon}$. The set of invariant distributions is a convex subset of $\M(M)$ and, just like in the deterministic case, we call its extreme points \emph{ergodic}. Such distributions describe random orbits of the system.

A consequence of the local property (\ref{local}) is that, given a family of stationary distributions
$\{ \mu_{\epsilon}\}_{\epsilon>0}$ of the corresponding $\T$, any weak accumulation point $\mu_0$ as $\epsilon \rightarrow 0$ is an $f$-invariant measure. Such measures are called \emph{zero noise limits}. The notion of stochastic stability is based on the idea that zero noise limits should be compatible with physical measures. 

\begin{defn}
Suppose $f \in \Dl(M)$ has some trapping region $U$ in which there exists a finite number of physical, say $\mu_1, \ldots, \mu_N$. We say that $f$ is stochastically stable (really, the pair $(f, U)$), if every zero noise limit $\mu_0$ is a convex combination of physical measures: $\mu_0 = \alpha_1 \mu_1 + \ldots + \alpha_N \mu_N$ for some non-negative $\alpha_1, \ldots, \alpha_N$.
\end{defn}

Traditionally, the notion of stochastic stability of an attractor assumed it to have a unique physical measure. The definition we have given above seems to be the natural generalisation, as no stronger property can be expected to hold in any greater generality. 
See Remark D.6. in \cite{BDV} for a discussion.

\begin{thmF}\label{stochstab}
Every $f$ in $MC$ is stochastically stable.
\end{thmF}

We remark that the apparent discrepancy between statistical and stochastic stability, revealed by comparing  Corollary D with Theorem F, is not of a profound nature. It merely reflects the strong definition of statistical stability considered. Should one have settled with the weaker form of statistical stability suggested in \cite{V}, one would obtain (quite trivially) that \emph{all} mostly contracting systems are statistically stable --- not only an open and dense set.

\section{Toolbox}\label{toolbox}

Whenever dealing with a normed vector space, $(V, \| \cdot \|)$ say, then $V(r)$ denotes the ball 
$V(r) = \{v\in V: \|v\| < r\}$ of radius $r$ centred at the origin. 

Given any submanifold $N \subset M$, we shall denote by $d^N(x, y)$ the intrinsic distance of points $x, y\in N$ defined as the infimum of arclengths of all smooth curves joining $x$ and $y$ \emph{inside} $N$. Similarily, for $x\in N$, $B_r^N(x)$ denotes the intrinsic ball $\{y\in N: d^N(x, y)<r\}$.

If we are dealing with a topological space, $X$ say, we may form the space $\M(X)$ of Borel probability measures on $X$. The space $\M(X)$ is always considered with the weak topology, in which convergence $\mu_n \rightarrow \mu$ is characterised by requiring that $\int \varphi d\mu_n \rightarrow \int \varphi \mu$ for every bounded continuous $\varphi: X \rightarrow \mathbb{R}$. If $K \subset X$ is a subset (compact or not), we sometimes use the notation $\M(K)$ to mean $\{\mu \in \M(X): \mu(K) = 1 \}$.

\subsection{Integral representation of measures}

We are going to integrate measure valued functions on many occasions. The following situation is then always understood: There are two Hausdorff spaces $X$ and 
$Y$, with $Y$ compact, and their associated spaces of Borel probability measures $\M(X),\ \M(Y)$ endowed with the weak topology. Thus $\M(X) \subset C_b^0(M)^*$ and $\M(Y) \subset C^0(Y)^*$, where $C_b^0(X)$ is the set of bounded continuous functions $X \rightarrow \mathbb{R}$. Suppose we are given some Borel probability $\mu \in \M(X)$ and a continuous map $\vartheta: X \rightarrow \M(Y)$. We define the measure 
$\int \negmedspace \vartheta d\mu \in \M(Y)$ by requiring
\begin{equation*}
\int \varphi \ d(\! \textstyle\int \negmedspace \displaystyle \vartheta d\mu) = \int\left( \int \varphi \ d\vartheta(x) \right) d\mu(x) 
\end{equation*}
for every continuous $\varphi:Y \rightarrow \mathbb{R}$.

Alternatively, given any Borel set $E\subset Y$, we have 
\[\textstyle \int \negmedspace \vartheta d\mu(E) = \displaystyle \int \vartheta(x)(E)d\mu(x).\]
Measurability of the map $x \mapsto \vartheta(x)(E)$ is established by dominated pointwise approximation of $\chi_E$ (the indicator function of $E$) by continuous functions.

In the language of convex analysis one would say that $\int \negmedspace \vartheta d\mu$ is the \emph{barycentre} of $\vartheta_* \mu$, or that $\vartheta_* \mu$ \emph{represents} 
$\int \negmedspace \vartheta d\mu$.

\begin{prop}\label{cont}
The mapping $\mu \mapsto \int \negmedspace \vartheta d\mu$ is continuous.
\end{prop}
\begin{proof}
 Take any continuous $\varphi:Y\rightarrow \mathbb{R}$. Continuity of $\vartheta$ means that 
$x \mapsto \int \varphi d\vartheta(x)$ is a bounded continuous function $X \rightarrow \mathbb{R}$. Call it $\tilde{\varphi}$. Then  $\int \varphi \ d(\int \negmedspace \vartheta d\mu) = \int \tilde{\varphi}d\mu$ by definition, so $\int \negmedspace \vartheta d\mu$ depends indeed continuously on $\mu$.
\end{proof}

\subsection{Admissible measures and carriers}

This section introduces the notion of admissible measures, the most important tool in this paper, used in the proof of all theorems. They should be thought of as non-invertible analogues of Gibbs-$u$ states (see \cite{PeSi} for definitions). Due to the non-invertibility of local diffeomorphisms, systems in $\PH$ do not have unstable foliations. Still, there is an invariant family of manifolds tangent to the unstable cone field. Tsujii \cite{T} defined admissible measures for partially hyperbolic maps with a $1$-dimensional unstable direction. They are smooth measures on an invariant family of unstable curves or, more generally, convex combinations of such. Great care has to be taken when extending his notion  to arbitrary dimension, due to the higher geometrical complexity.

\subsubsection{Admissible manifolds}

We follow the approach in \cite{ABV} for defining an invariant family of manifolds of bounded curvature.

A $C^1$ embedded $u$-dimensional submanifold $N \subset M$ is said to be tangent to $S^u$ if $T_x N \subset S_x^u$ for every $x \in N$. Further, we say that the tangent bundle of $N$ is Lipschitz continuous if 
$N \ni x \mapsto T_x N \subset G^u M$ is a Lipschitz continuous section of the Grassmannian bundle (see Section \ref{Grass}). The Lipschitz variation may be quantified by considering the variation of $T_x N$ in exponential charts. More precisely, we choose some small $\delta$ so that, at every $x \in M$,  the exponential map $\exp_x: T_x M(\delta) \rightarrow M$ is a diffeomorphism; and denote by $\tilde{N}_x$ the preimage of $N$ under $\exp_x$. Each point $y \in B_{\delta}(x)$ corresponds to a point $\exp_x^{-1}(y)$ in $T_x M (\delta)$ which we denote by $\tilde{y}$. In particular, $\tilde{x}$ is the zero element in $T_x M$. 

For every $\tilde{y} \in \tilde{N}_x$, there is a unique map $A_x(y) : T_x N \rightarrow E_x^c$ whose graph is parallel to $T_{\tilde{y}} \tilde{N}_x$. 
We say that the tangent bundle of $N$ is $K$-Lipschitz continuous at $x\in N$ if $\|A_x(y)\| \leq K d^N(x, y)$ for every $\tilde{y} \in \tilde{N}_x$. 
Furthermore, the tangent bundle of $N$ is $K$-Lipschitz if it is $K$-Lipschitz at every $x$.

\begin{prop}\label{curvature}
Let $f$ be partially hyperbolic. There exists a neighbourhood $\mathcal{U}$ of $f$ and $K_0>0$ such that for any $g$ in $\mathcal{U}$, and any
$C^1$ embedded disc $N$ tangent to $S^u$ with $K_0$-Lipschitz 
tangent bundle, the tangent bundle of $g^n(N)$ has Lipshitz constant smaller than $K_0$ for every $n>n_0$.
\end{prop}

\begin{proof}
 Fix some $x \in N$ and let $\tilde{f}^n$ be the map from a neighbourhood 
$\tilde{U}_x$ of the origin in 
$T_x M$ to a neighbourhood $\tilde{U}_{f^n(x)}$ of the origin in $T_{f^n(x)}$, given by
\[\tilde{f}^n = \exp_{f^n(x)}^{-1} \circ f \circ \exp_x  .\]
We identify $T\tilde{U}_x$ with $T_x M$ (and likewise $T\tilde{U}_{f^n(x)}$ with $T_{f^n(x)}M$) by translation. Let $P$ be the constant field in $\tilde{U}_{f^n(x)}$ associating to each 
$z\in \tilde{U}_{f^n(x)}$ the subspace $T_{f^n(x)}f^n(N)$. We pull-back $P$ through $\tilde{f}^n$ to obtain another field $Q$ in $\tilde{U}_x$. Thus
\[D\tilde{f}^n(\tilde{y}) Q(\tilde{y}) = T_{f^n(x)}f^n(N) \quad \forall \tilde{y} \in \tilde{U}_x.\]
To each $\tilde{y}\in \tilde{U}_x$ is associated a unique linear map $B_x(y): T_x N \rightarrow E_x^c$ such that $Q(\tilde{y})$ is the graph of $B_x(y)$. Since $f$ is $C^2$, there is some $C_0>0$, uniform in some neighbourhood $\mathcal{U}$ of $f$, such that 
\[\|B_x(y)\| \leq C_0 d(x, y).\]

Suppose the tangent bundle of $N$ is $K$-Lipschitz for some $K$. That is, $T_y\tilde{N}_x$ is the graph of a uniquely defined linear map $A_x(y):T_x N \rightarrow E_x^c$, satisfying
\[\|A_x(y)\|\leq K d^N(x, y).\]
Therefore it is also the graph of the map $\tilde{A}_x(y): Q(y)\rightarrow E_x^c$ given by
\[\tilde{A}_x(y) = A_x(y) - B_x(y).\]

We wish to estimate the norm of $A_{f^n(x)}(f^n(y))$, i.e. the linear map from $T_{f^n(x)}f^n(N)$ to 
$E_{f^n(x)}^c$ whose graph coincides with $T_{\tilde{f}^n(\tilde{y})}\tilde{N}_{f^n(x)}$. Note that
\[A_{f^n(x)}(\tilde{f}^n(\tilde{y})) = D\tilde{f}^n(\tilde{y})_{\vert E_x^c} \tilde{A}_x(y)(D\tilde{f}^n(\tilde{y})_{\vert T_x N})^{-1}\]
so it follows from (\ref{domination}) that 
\begin{align*}
\|A_{f^n(x)}(\tilde{f}^n(y))\| &\leq \tau^{n-n_0} \|A_x(y) - B_x(y)\| \\
&\leq \tau^{n-n_0}\left( K d^N(x,y) + C_0 d(x, y)\right) \\
&\leq \tau^{n-n_0}(K+C_0) d^N(x, y) \\
&\leq (\tau^{n-n_0})^2(K+C_0) d^{f^n(N)}(f^n(x), f^n(y)).
\end{align*}
The proposition follows by taking $K_0 > C_0 \frac{(\tau^{n-n_0})^2}{1-(\tau^{n-n_0})^2}$.
\end{proof}

We fix a value of $K_0$ once and for all as in Proposition \ref{curvature}. 

\begin{defn}
We say that a $u$-dimensional $C^1$ embedded manifold is admissible if it is tangent to $S^u$,  has a $K_0$-Lipschitz tangent bundle, or is the iterate of such under $f^k$, $k=1, \ldots, n_0$.
\end{defn}

By Proposition \ref{curvature}, the set of admissible manifolds is invariant under iterates of $f$. Actually, there is some $C^2$ neighbourhood $\mathcal{U}$ of $f$ such that the set of admissible manifolds is invariant under every $g \in \mathcal{U}$. This 'rigidity' property will become important in the study of small perturbations of $f$, bot of random and deterministic type. Let $m_N$ be Lebesgue measure on some admissible manifold. One may wonder what the possible weak accumulation points of the sequence 
$\frac{1}{n} \sum_{k=0}^{n-1} f_*^k m_N$ are. This is where admissible measures enter the scene. They are convex combinations of smooth measures on admissible manifolds. However, it is not practical to work with the space of all admissible manifolds, but only consider a very particular kind. These will be called carriers, because their lot in life is to `carry' admissible measures. 

\subsubsection{The Grassmannian bundle} \label{Grass}

Recall that the $u$-dimensional Grassmannian manifold over a vector space $V$ is the set $G^u(V)$ of $u$-dimensional subspaces of $V$. It can be turned into a compact smooth $u(n-u)$-dimensional manifold by modelling it over the space $L(\mathbb{R}^u, \mathbb{R}^{n-u})$ of linear maps from $\mathbb{R}^u$ to $\mathbb{R}^{n-u}$. Namely, if $H \in G^u(V)$, then 
\begin{equation}
 L(\mathbb{R}^u, \mathbb{R}^{n-u}) 
\simeq L(H, H^{\perp}) \ni A \mapsto \Graph(A) 
\in G^u(V)
\end{equation}
defines (the inverse of) a local chart of $G^u(V)$ around $H$; here $H^{\perp}$ is any subspace of $V$, complementary to $H$. Let $G^u M = \bigcup_{x \in M} G^u(T_x M)$. It may be considered as a bundle over $M$. 
Indeed, let $p:G^u M \rightarrow M$ be the natural projection and $(U_0, \varphi_0)$ some chart on $M$.
We define a bundle chart $\underline{\varphi}_0: p^{-1}(U_0) \rightarrow U_0 \times G^u(\mathbb{R}^u)$ by 
$\underline{\varphi}_0(x, h) = (x, D\varphi_0(x)h)$. The topology given on $G^u M$ is then locally the product topology of $U_0 \times G^u(\mathbb{R}^u)$ induced by $\underline{\varphi}_0$.
In this way $G^u M$ becomes a compact manifold and the unstable conefield $S^u$ is a closed subset. 
We fix a number $r_0$, small enough for the exponential map to be a diffeomorphism on $r_0$-balls in $T_x M$ at every $x \in M$. We will impose further conditions on the value of $r_0$ later on.

\subsubsection{Carriers}

Having understood the notion of admissible manifolds and Grassmannian bundle, the time is now ripe for making the notion of a carrier precise.

\begin{defn}
 A carrier is a quadruple $\Gamma = (r, x, h, \psi)$, where
\begin{itemize}
 \item $r<r_0$ is a positive real number (called the radius of $\Gamma$)
\item $x$ a point in the trapping region $U$ (called the centre of $\Gamma$)
\item $h \subset S_x^u$ is a $u$-dimensional subspace of $T_x M$ (called the direction of $\Gamma$)
\item $\psi: h(r) \rightarrow E_x^c$ is a $C^1$ map such that 
\begin{enumerate}
 \item $\psi(0)=0$,
\item $D \psi(0) = 0$
\item $\exp_x \Graph (\psi)$ is an admissible manifold.
\end{enumerate}
\end{itemize}
\end{defn}

Recall the notation introduced in the introduction of Section \ref{toolbox}: $h(r)$ is the ball $\{v \in h: \|v\| < r \}$. Provided that the number $r_0$ is small, the manifold $\exp_x \Graph (\psi)$ may be thought of as an almost round $u$-dimensional disc of radius $r$, centred at $x$ and tangent to $h$ at $x$. The jargon we will adopt is that we identify $\Gamma$ with $\exp_x \Graph(\psi)$. So a carrier $\Gamma$ is in fact to be thought of as a special kind of admissible manifold, the quadruple $(r, x, h, \psi)$ being its coordinates. It should be clear that if $r_0$ is sufficiently small, so that the carriers are flat enough, there is only one possible centre and, consequently, only one coordinate description of a given `carrier-manifold'. 

The space of all carriers will be denoted by $\K$ and divided into strata $\K(a)$, consisting of carriers with radius $a$. It will be given a topology in section \ref{topology}, turning it into a separable metrizable space with each stratum $\K(a)$ being a compact subset.

\subsubsection{Simple admissible measures}

Consider some carrier $\Gamma$ with coordinates $(x, r, h, \psi)$. Let $\omega$ denote the volume form on $M$ derived from the Riemannian metric. Then, letting $i_{\Gamma}:\Gamma \rightarrow M$ denote inclusion, we obtain an induced volume form $\omega_{\Gamma}:= i_{\Gamma}^*\omega$ on $\Gamma$. We denote by $\vert \Gamma \vert$ the total mass 
$\int_{\Gamma} \omega_{\Gamma}$ of $\Gamma$ and write $(\Gamma, 1)$ for the normalised volume on $\Gamma$: 
\begin{equation}\label{Lebesgue on Gamma}
(\Gamma, 1)(E) = \frac{1}{\vert \Gamma \vert} \int_{E\cap \Gamma} \omega_{\Gamma} 
\end{equation}
for every measurable $E\subset M$. Thus $\{(\Gamma, 1): \Gamma \in \K\}$ is the family of normalised Lebesgue measure on carriers. We wish to enlarge this family by considering absolutely continuous measures with bounded densities. Suppose $\phi: \Gamma \rightarrow \mathbb{R}$ is a non-negative integrable (density) function. Then we may define the measure $(\Gamma, \phi)$ by
\begin{equation}\label{gammaphi}
(\Gamma, \phi)(E) = \frac{1}{\vert \Gamma \vert} \int_{E\cap\Gamma} \phi \ \omega_{\Gamma}
\end{equation}
on Borel sets $E\subset M$. The notation $(\Gamma, 1)$ for the measure (\ref{Lebesgue on Gamma}) should now be transparent.

\begin{defn}
 A simple admissible measure is a quintuple $(r, x, h, \psi, \phi)$ such that $\Gamma = (r, x, h, \psi)$ 
is a carrier and $\phi: \Gamma \rightarrow \mathbb{R}$ is a Borel function satisfying 
\begin{itemize}
\item
$\frac{1}{\vert \Gamma \vert} \int_{\Gamma} \phi \ \omega_{\Gamma} = 1$ 
\item $\log \phi$ is bounded. 
\end{itemize}
\end{defn}

We seldom refer to a simple admissible measure explicitly as a quintuple, but more frequently as a pair $(\Gamma, \phi)$. It is then understood that $\Gamma$ is a carrier, say $\Gamma = (r, x, h, \psi)$, and $(\Gamma, \phi)$ should then be interpreted as $(r, x, h, \psi, \phi)$. Just like a carrier, a simple admissible measure also has a radius, a centre and a direction, given in the obvious way. By now, it should not come as a surprise that we identify a simple admissible measure $(\Gamma, \phi)$ with the measure 
\[ E \mapsto \frac{1}{\vert \Gamma \vert} \int_{\Gamma \cap E} \phi \ \omega_{\Gamma}\]
 ($E \subset M$ is any Borel set).

The set of all simple admissible measures is denoted by $\A$. It can harmlessly be thought of as a subset of $\M(M)$. It also splits into strata $\A(a)$, consisting of simple admissible measures of radius $a$. Furthermore, each strata is the union of a nested family of sets 
\[\A(a, C) = \{(r, x, h, \psi, \phi) \in \A: r=a \text{ and } 
C^{-1} \leq \phi \leq C \}\]
of decreasing level of regularity. We are going to proove that, seen as a subset of $\M(M)$, each $\A(a,C)$ is compact. The proceedure is rather prolix: we define a topology on $\A$ using a nested fibre construction; then prove that thus endowed, each $\A(a,C)$ is compact. Finally we observe that the inclusion $\A(a, C) \rightarrow \M(M)$ is continuous.

\subsubsection{An interlude into the heuristics of admissible measures}\label{interlude}
Let us pause for a moment to take a peep on what is to come. It is clear that $\K$ is not an $f$-invariant family. Indeed, an iterate $f^n(\Gamma_0)$ of a carrier $\Gamma_0$ is generally some large unshapely immersed disc that may intersect itself  and is quite far from being round. For the same reason, $\A$ cannot be invariant under $f_*$. Still, it is quite clear that $f^n(\Gamma_0)$ is a union of carriers, although obviously not a disjoint one. But there is some hope that $f_*^n (\Gamma_0, 1)$ has an integral representation on simple admissible measures:
\begin{equation} \label{representation}
f_*^n (\Gamma_0, 1) = \int_{\A} (\Gamma, \phi)  \ d\! \umu(\Gamma, \phi) 
\end{equation}
where $\umu$ is some measure on $\A$. And so it is indeed. But since $f^n(\Gamma)$ is \emph{not} a disjoint union of carriers, the measure $\umu$ cannot be atomic. This corresponds to the fact that one cannot cut a large disc (of dimension at least $2$) into a number of smaller ones. (The remaining objects would not look like round discs, but bear more resemblance to half moons or, even worse, splinters of broken porcelain.) The measure $\umu$, at least in the way we will construct it in section \ref{technique}, is not supported on a single strata $\A(a)$. However, provided that $n$ is large, $\umu$ will give weight nearly $1$ to some specified strata $\A(a)$. This allows us to prove that every accumulation point of $\frac{1}{n} \sum_{k=0}^{n-1} f_*^k (\Gamma, 1)$ has an integral representation of the form (\ref{representation}), and with $\umu$ supported on some $\A(a)$ --- a fact of great importance for the proofs of all results in this work.

\subsubsection{Topology on $\A$ and $\K$}\label{topology}

We use bundel constructions to topologise $\A$ and $\K$. The idea is that, locally, $\K$ should look like a subset of the product space
\[ \mathbb{R} \times M \times G^u(\mathbb{R}^n) \times C_b^1(\mathbb{D}^u, \mathbb{R}^{n-u}), \]
$C_b^1(\mathbb{D}^u, \mathbb{R}^{n-u})$ being the space of bounded $C^1$ maps from the unit $u
$-dimensional disc $\mathbb{D}^u$ to $\mathbb{R}^{n-u}$ and whose derivatives are also bounded. It is considered with the usual $C^1$ topology.
  
The topology of $\A$ is to take (locally) the form of  
\[\K \times L_w^2(\mathbb{D}^u). \]
Here $L_w^2(\mathbb{D}^u)$ is the space of square integrable Borel functions $:\mathbb{D}^u \rightarrow \mathbb{R}$ endowed with the weak topology, in which convergence $\phi_n \rightarrow \phi$ is characterised by requiring that $\int \phi_n \psi d(\Gamma, 1) \rightarrow \int \phi \psi d(\Gamma, 1)$ for every $\psi \in L_w^2(\Gamma)$. 

To carry out the construction explicitly, let $I$ be the interval $(0, r_0)$ and consider the sets 
\begin{align*}
\tilde{\K} &= \bigcup_{r \in I} \quad \bigcup_{x \in M}\quad  \bigcup_{h \in T_x M}
\ \bigcup_{\psi \in C_b^1(h(r), E_x^c)} (r, x, h, \psi), \\
 \tilde{\A} &= \bigcup_{\Gamma \in \tilde{\K}} L_w^2(\Gamma).
\end{align*}
The difference between $\tilde{\K}$ and $\K$ is that for a quintuple $(r, x, h, \psi)$ to belong to $\tilde{\K}$ it does not have to satisfy items (1)-(3) in the definition of carriers. We shall define topologies on $\tilde{\K}$ and $\tilde{\A}$ and consider $\K$ and $\A$ as subsets. 

Naturally, we give 
\[ \bigcup_{r \in I} \ \bigcup_{x \in M}\  \bigcup_{h \in T_x M} (r, x, h) \]
the topology of $I \times G^u M$. Thus we write
\[\tilde{\K} = \bigcup_{(r, x, h) \in I \times G^u M} C_b^1(h(r), E_x^c), \]
and intend to consider $\tilde{K}$ as a vector bundle over $I \times G^u M$.To define the bundle charts, fix $(r_0, x_0, h_0) \in I \times G^u M$ and take some local chart $(V_0, \varphi_0)$ of $M$ around $x_0$. Let $\underline{p}$ be the canonical projection $\tilde{\K} \rightarrow I \times G^u M$ taking $(r, x, h, \psi)$ into $(r, x, h)$. Write 
\begin{align}\nonumber
 & H_0  = D\varphi(x_0) h_0, \\
 & D_0  = D \varphi(x_0) h_0(r_0), \label{D0}\\
 & E_0  = D\varphi(x_0) E_{x_0}^c, \nonumber
\end{align}
so that $D_0$ is a $u$-dimensional disc (ellipsoid) in $\mathbb{R}^n$. Each fibre $C_b^1(h(r), E_x^c)$ can be modelled over $C_b^1(D_0, E_0)$. To this end we must define a map $\Psi$ from $\underline{p}^{-1}(I \times p^{-1}(V_0))$ to 
$I\times p^{-1}(V_0) \times C^1(D_0, H_0)$ such that
\[
\xymatrix{
\tilde{\K}  \supset   \underline{p}^{-1}(I \times p^{-1}(V_0)) \ar[d]_{\underline{p}} \ar[r]^{\Psi} & I \times p^{-1}(V_0) \times C_b^1(D_0, E_0) \ar[dl]^{\pi} \\ 
                           I \times p^{-1}(V_0)       }
\]
commutes. Thus $\Psi(r, x, h, \psi)$ should take the form $(r, x, h, \Psi_{(r, x, h)} \psi)$ for some continuous linear map $\Psi_{(r, x, h)} : C_b^1(h(r), E_x^c) \rightarrow C_b^1(D_0, E_0)$. Then we take neighbourhoods of $(r_0, x_0, h_0, \psi_0)$ in $\tilde{\K}$ to be just the preimages, under $\Psi$, of neighbourhoods of $(r_0, x_0, h_0, \Psi_{(r_0, x_0, h_0)} \psi_0)$ in the product topology of $I \times p^{-1}(V_0) \times C_b^1(D_0, E_0)$.

Given any $(r, x, h) \in V_0$ there is a unique linear map $A_{(x, h)}$ such that $D\varphi_0(x) h = \Graph A_{(x, h)}$. Let $\overline{A}_{(x, h)} : H_0 \rightarrow \mathbb{R}^n$ be the map $v \mapsto (v, A_{(x, h)}v)$. 
We define a linear map $T_{(r, x, h)}:H_0 \rightarrow h$ by
\[T_{(r, x, h)} v =  \frac{r \|v\|_{x_0}}{r_0} \cdot \frac{D \varphi(x)^{-1} \overline{A}_{(x, h)} v}
{\|D \varphi(x)^{-1} \overline{A}_{(x, h)} v \|_{x}},
\]
mapping $D_0$ into $h(r)$.
As $H_0$ and $E_0$ are complementary spaces, we may identify $\mathbb{R}^n$ with the product $H_0 \times E_0$. Let $\pi_{E_0}$ be the projection to the second coordinate. Now $\Psi$ is defined by letting 
\[
 \Psi_{(r, x, h)} \psi (v) = \pi_{E_0} D\varphi(x) (T_{(r, x, h)}v, \psi(T_{(r, x, h)} v))
\]
for each $\psi$ in $C_b^1(h(r), E_x^c)$. 

If $(r_1, x_1, h_1)$ is another point in $I \times G^u M$ we pick a chart $(V_1, \psi_1)$ around $x_1$ and produce another bundle chart 
\[ \Psi':\underline{p}^{-1}(I \times p^{-1}(V_1)) \rightarrow I \times p^{-1}(V_0)\times C_b^1(D_1, E_1) \] in the same way.
We leave it to the reader to verify that if $V_0 \cap V_1 \neq \emptyset$, then
\[
 \Psi' \Psi^{-1}: I \times p^{-1}(V_0 \cap V_1)\times C_b^1(D_0, E_0) \rightarrow I\times p^{-1}(V_0 \cap V_1) \times C_b^1(D_1, E_1)
\]
is indeed a fibre preserving homeomorphism.

\begin{prop}\label{compactK}
 $\K(a)$ is compact for every $a \in (0, r_0)$.
\end{prop}

\begin{proof}
One may observe quite generally that if $\pi: F \rightarrow B$ is a fibre bundle over a compact base $B$ and $C \subset F$ a subset such that
\begin{enumerate}
 \item $C$ is closed,
\item $C \cap \pi^{-1}(p)$ is compact for every $p \in B$,
\end{enumerate}
then $C$ is compact. The proof of Proposition \ref{compactK} follows by taking $B = \{a\} \times G^u M$, $F = \tilde{\K}(a)$ and $C = \K(a)$. $\K(a)$ is closed because having $K_0$-Lipschitz tangent bundle is a closed property under $C^1$ convergence, and each $p^{-1} ((a, x, h))\cap \K(a)$ is compact by the Arzel\`{a} Ascoli theorem.
\end{proof}

We proceed to put a topology on $\tilde{\A}$ as follows. Fix $\Gamma_0 = (r_0, x_0, h_0, \psi_0) \in \tilde{\K}$ and let $W$ be some small neighbourhood of $\Gamma_0$. We write $q$ for the canonical projection $\tilde{\A} \rightarrow \tilde{\K}$. The topology we give on $\tilde{\A}$ is, again, locally a product topology, obtained by turning $q: \tilde{\A} \rightarrow \tilde{\K}$ into a fibre bundle. Each fiber $L_w^2(\Gamma)$ is isomorphic, via identification $h(r) \ni v \mapsto \exp_x (v, \psi(v))$, to $L_w^2(D_0)$. (Here $D_0$ is defined as in (\ref{D0})). Thus a map $\Phi$ from $q^{-1}(W)$ to $W \times L_w^2(D_0)$ must be defined so that 
\[
 \xymatrix{
\tilde{\A}  \supset   q^{-1}(W) \ar[d]_{q} \ar[r]^{\Phi} & V\times L_w^2(D_0) \ar[dl]^{\pi} \\ 
                           W       }
\]
commutes, i.e. $\Phi(r, x, h, \psi, \phi)$ should take the form 
$(r, x, h, \psi, \Phi_{(r, x, h, \psi)} \phi)$.
The natural choice here is 
$\Phi_{(r, x, h, \psi)}\phi (v) = \phi(\exp_x (T_{(r, x, h)} v , \psi(T_{(r, x, h)} v)))$. One readily verifies that if $\Phi'$ is difined analogously to $\Phi$ over some neighbourhood $W'$ of a carrier $\Gamma_1$ such that $W \cap W' \neq \emptyset$, then 
\[
 \Phi' \Phi^{-1}: (W\cap W') \times L_w^2(\Gamma_0) \rightarrow (W\cap W') \times L_w^2(\Gamma_1)
\]
is a fibre preserving homeomorphism.

\begin{prop}
 Every $\A(a, C)$ is compact, $a \in I$ and $C>0$.
\end{prop}

\begin{proof}
 We apply the same argument as in Proposition \ref{compactK}. All we need to check is that each set
\[\A(a, C)_{\vert \Gamma} = \{\phi \in L_w^2(\Gamma): C^{-1} \leq \phi \leq C \text{ and } \int \phi \ d(\Gamma, 1) = 1 \} \]
is compact. Note that $\A(a, C)_{\vert \Gamma}$ is contained in the ball 
\[B_{C^2}:=\{ \phi \in L^2(\Gamma): \|\phi\|_2 \leq C^2 \}.\]
Since $L^2(\Gamma)$ is a Hilbert space, it is isomorphic to its dual space $(L^2(\Gamma))^*$ and the weak topology on $L^2(\Gamma)$ corresponds to the weak* topology on $(L^2 (\Gamma))^*$.
Hence, by the Banach-Alaoglu theorem, $B_{C^2}$ is compact. Consequently, 
every sequence $\phi_n \in \A(a, C)_{\vert \Gamma}$ has a weak accumulation point, i.e. $\int \phi_{n_j} \varphi d(\Gamma, 1) \rightarrow \int \phi \varphi d(\Gamma, 1)$ for some subsequence $\phi_{n_j}$ and every $\varphi \in L^2(\Gamma)$. In particular, 
\[
\lim_{j \rightarrow \infty} \frac{\int \chi_E \phi_{n_j} d(\Gamma, 1)}{(\Gamma,1)(E)} 
= \frac{\int_E \phi d(\Gamma, 1)}{(\Gamma, 1)(E)} \in [C^{-1}, C]
\]
for every Borel set $E \subset M$. Hence
$C^{-1} \leq \phi \leq C$. Taking $\varphi = 1$ proves that $\int \phi \ d(\Gamma, 1) = 1$, so $\A(a, C)_{\vert \Gamma}$ is indeed compact.
\end{proof}

\subsubsection{Admissible Measures}\label{admissmeas}

Let $\iota:\A \rightarrow \M(M)$ be the map that associates a quintuple to its corresponding measure. It is clear that $\iota$ is a continuous injection. Therefore, each regularity level $\A(a, C)$ of each strata $\A(a)$, $a<r_0$, corresponds to a compact set $\iota(\A(a, C)) \subset \M(M)$.

Consider the space
$\mathcal{M}(\mathbf{A})$  of Borel probability measures on
$\mathbf{A}$, endowed with the weak topology of measures.
We define the map 
\begin{eqnarray*}
 \uiota: \M(\A) &\rightarrow & \M(M) \\
\umu & \mapsto & \int \iota \ d\umu.
\end{eqnarray*}
That is, $\uiota(\umu)$ is given by the Fubini-like relation
\begin{equation}\label{fubini}
\uiota (\umu)(E) = \int_{\A} (\Gamma, \phi)(E) \ d \underline{\mu}(\Gamma, \phi)
\end{equation}
for every Borel set $E\subset M$.

\begin{defn}
We say that $\underline{\mu}$ is a lift of $\mu$ if $\uiota(\umu) = \mu$.
A measure $\mu$ in $\mathcal{M}(M)$ is said to be admissible if it
has some lift in $\M(\mathbf{A})$.
\end{defn}

It is useful to think of admissible measures as being convex combinations, in an ample sense, of simple admissible measures.

The space of admissible measures will be denoted by $\AM$. Furthermore, 
we write $\AM(a, C)$ for the set of admissible measure that have a lift
supported in $\A(a, C)$. Thus
\begin{itemize}
 \item $\AM = \uiota(\M(\A))$ and 
\item $\AM(a, C) = \uiota(\M(\A(a, C)))$
\end{itemize}
for every $0<a<r_0,\ C>0$.
Since $\iota:\A \rightarrow \M(M)$ is continuous, so is $\uiota$ (Proposition \ref{cont}). Therefore, admissible measures of fixed radius and bounded regularity levels form compact spaces:

\begin{prop}\label{admismeascomp}
$\AM(a, C)$ is compact for every $0<a < r_0$ and $C>0$.
\end{prop}

\begin{proof}
The image of any continuous map from a compact space to a Hausdorff space is compact.
\end{proof}

Contrary to $\iota$, $\uiota$ is not injective. 
This means that lifts are not unique. An easy illustration of this fact is to consider
Lebesgue measure $m$ on the circle. Here $\A$ is understood to be the collection of all measures equivalent to Lebesgue restricted to some interval, and whose densities density is bounded away from zero and infinity.
 
\begin{ex} \label{atomiclift}
We may partition the circle into any
finite  number of curves, say $\gamma_1, \ldots, \gamma_k$. Writing $\alpha_i = m(\gamma_i)$ and $m_{\vert \gamma_i}$ for normalised restrictions, we get the representation
\[
m = \sum_{i=1}^k \alpha_i m_{\vert \gamma_i}.
\]
That is, $m$ has the lift 
\[\underline{m}_1 = \sum_{i=1}^{k} \alpha_i \delta_{m_{\vert \gamma_i}}.\]
\end{ex}

The lift of $\underline{m}_1$ thus obtained is \emph{atomic}: it is a convex combination of Dirac measures. Each term corresponds to a line segment obtained by cutting the circle. As mentioned in section \ref{interlude}, this cutting business cannot be used in higher dimeinsions as it alters the geometry of objects too much, and quite a different philosophy must be adopted.

\begin{ex}\label{nonatomiclift}
Consider Lebesgue measure $m$ on the circle, just like in Example \ref{atomiclift}. Fix some small number $a>0$ and, for every $x \in S^1$, denote by $m_x$ the normalised restriction of $m$ to the interval $(x-a, x+a)$. The measure $m$ can then be expressed by the relation
\[
m = \int m_x \ dm(x).
\]
In this case we obtain the lift $\underline{m}_2 = \xi_* m$, where 
$\xi: S^1 \rightarrow \mathbf{A}(a)$ is the map taking $x$ to $m_x$.
\end{ex}

\subsubsection{A disintegration technique}\label{technique}

The lift $\underline{m}_2$ in the previous example is in certain ways superior to $\underline{m}_1$. One reason is that it perfectly reflects $m$, in the sense that the distribution of the centre of carriers is given by $m$ itself. More importantly, lifts analogous to $\underline{m}_2$ can be constructed in higher dimension, whereas atomic lifts like $\underline{m}_1$ cannot. Below we set forth a general scheme to produce non-atomic lifts to smooth measures in a more general setting. We will be able to conclude

\begin{prop}\label{AM invariant}
There exists a neighbourhood $\mathcal{U}$ of $f$ in $\PH$ such that $\AM$ is invariant under every $g \in \mathcal{U}$.
\end{prop}

This is a rather curious fact, since $\A$ is far from being invariant. For better comprehension, 
we illustrate the technique by a toy model on the interval $(0, \infty)$. Once understood, 
the general construction is a straightforward adaption, although the underlying idea gets a bit obscured by heavy notation.

\begin{ex}\label{onedimdisint}
Let $m$ denote Lebesgue measure on $I:=(0, \infty)$ and $R:I \rightarrow \mathbb{R}$ be a function defined by $R(x)= x/2$. Consider the family $\{m_x\}_{x\in I}$ of normalised Lebesgue measure on $I_x := (x-R(x), x + R(x))$. We shall find a family of densities densities $\phi_x: I_x \rightarrow \mathbb{R}$ with  $\int_{I_x} \phi_x\  dm_x = 1$, and a weight $\rho: I\rightarrow \mathbb{R}$,  such that 
\begin{equation*}
\int (\phi_x m_x) \  d(\rho m)(x) = m.
\end{equation*}
That is done by first finding any family $\tilde{\phi}_x:I_x \rightarrow \mathbb{R}$ satisfying
\begin{equation*}
\int (\tilde{\phi}_x m_x) \ dm(x) = m;
\end{equation*}
then take $\rho(x) = \int \tilde{\phi}_x dm_x$ and normalise $\phi_x = (\rho(x))^{-1}\tilde{\phi}_x$.

Let $V_y= \{x\in I: \vert y-x\vert<R(x)\}  = (\frac{2 y}{3}, 2y)$. The trick is to take
\[
\tilde{\phi}_x(y) = \frac{m(I_x)}{m(V_y)} = \frac{3x}{4y}
\]
so that each $\frac{\tilde{\phi}_x}{m(I_x)}$ gives the same value at $y$, whenever $x \in V_y$.

We have 
\[
\tilde{\phi}_x m_x (E) = \int_{I_x} \frac{\tilde{\phi}_x(y) \chi_E(y)}{m(I_x)}dm(y)
\]
for any Borel set $E \subset \mathbb{R}$. (Here $\chi_E$ denotes the indicator function of $E$.)

Hence, by Fubini's Theorem, 
\begin{equation*}
\begin{split}
 \int \tilde{\phi}_x m_x(E)\ dm(x) = \int_I \left( 
\int_{I_x} \frac{\tilde{\phi}_x(y) \chi_E(y)}{m(I_x)}dm(y) \right)dm(x) \\
= \int_I \left( \int_{V_y} \frac{\chi_E(y)}{m(V_y)} dm(x)\right) dm(y)
=\int \chi_E(y)\ dm(y) = m(E)
\end{split}
\end{equation*}
as required. One may check that $\rho(x) \equiv \frac{3}{4} \log3$, so the family $\{\phi_x\}_{x \in I }$ is given by $\phi_x(y) = \frac{x}{y \log3}$.
\end{ex}

Now suppose that 
$(\Gamma, \phi) = (r, p, h, \psi, \phi)$ is some simple admissible measure. We shall prove that, although iterates of $\Gamma$ under $f^n$ are not carriers, push-forwards of $(\Gamma, \phi)$ under $f_*^n$ are admissible measures. A lift of $f_*^n (\Gamma, \phi)$ will be given explicitly and we will see in section \ref{ILBL} that this choice of lift has some extra good properties.

Consider some iterate $f^n(\Gamma)$ of the original carrier. Since $f$ is a local diffeomorphism, $f_{\vert \Gamma}$ is an immersion. However, if $n$ is large, it may happen that $f_{\vert \Gamma}$ is not injective. In particular, $f^n(\Gamma)$ need not be a submanifold in the strict sense of the word. Nevertheless, we shall associate, to each $x$ in $\Gamma$, a carrier denoted by $\Gamma_x$, such that 
$\bigcup_{x \in \Gamma} \Gamma_x = f^n(\Gamma)$.

For this purpose we define a new metric $\langle \cdot, \cdot \rangle^{\Gamma,n}$ on $\Gamma$, given by the pullback of the Riemannian metric through $f^n$: 
For $x\in \Gamma$ and $u, v \in T_x\Gamma$, set
\[ \langle u, v \rangle^{\Gamma, n} = \langle Df^n(x) u, Df^n(x) v \rangle\]
and let $d^{\Gamma,n}(x, y)$ be the distance on $\Gamma$ calculated using $\langle \cdot, \cdot \rangle^{\Gamma, n}$. We define a radius function 
\[R_a(x) = \min\{a,\textstyle \frac{1}{2}d^{\Gamma,n}(x, \partial \Gamma)\},\]
where $a$ is some small number in the interval $(0, r_0)$.
This choice makes $R_a$ Lipschitz continuous with constant $\frac{1}{2}$. 

We associate, to every $x \in \Gamma$, the space $h_x = Df^n(x) T_x \Gamma \in S_{f^n(x)}^u$, and identify $T_{f^n(x)} M$ with $h_x \times E_{f^n(x)}^c$. Since $R_a(x)$ is much smaller than $d^{\Gamma,n}(x, \partial \Gamma)$, there is some small connected neighbourhood $W_x$ of $x$ in $\Gamma$ such that $\exp_{f^n(x)}^{-1} f^n(W_x)$ is the graph of some $C^1$ map $\psi_x: h_x(R_a(x)) \rightarrow E_{f^n(x)}^c$. By Proposition \ref{curvature}, $\Gamma_x := (R_a(x), f^n(x), h_x, \psi_x) \subset f^n(\Gamma)$ is a carrier. Clearly
\[ f^n(\Gamma) = \bigcup_{x \in \Gamma} \Gamma_x .\]

Our goal is to find densities $\phi_x$ associated to each carrier $\Gamma_y$ and $\rho$ on $\Gamma$ such that, if $\xi$ is the map
\[\Gamma\ni x \mapsto (\Gamma_x, \phi_x) \in \A,\]
then $\xi_*(\Gamma, \rho)$ is a lift of $f_*^n (\Gamma, \phi)$. The construction of such densities will be made in three steps.

\begin{description}
 \item[Step 1]
\end{description}

A neighbourhood 
$V_y = \{x \in \Gamma: y \in W_z\}$ is assigned to every $y$ in $\Gamma$. Let $\tilde{\phi}_x: W_x \rightarrow \mathbb{R}$ be the family of densities given by 
\[\tilde{\phi}_x(y) = \frac{\phi(y)}{(\Gamma, 1)(V_y)} .\]
We claim that, given any Borel set $E \subset M$, we have
\[\int_{\Gamma} \left( \int_{E\cap W_x} \tilde{\phi}_x \ d(\Gamma, 1) \right) d(\Gamma, 1)(x) = (\Gamma, \phi)(E).\]
Indeed, 
\begin{eqnarray*}
&&\int_{\Gamma} \left( \int_{W_x} \tilde{\phi}_x(y) \chi_E(y) \ d(\Gamma, 1)(y) \right) d(\Gamma, 1)(x) \\
&=& \int_{\Gamma} \left( \int_{V_y} \tilde{\phi}_x(y)\chi_E (y) \ d(\Gamma, 1)(x) \right) d(\Gamma, 1)(y) \\
&=& \int_{\Gamma} \left(\int_{V_y} \ d(\Gamma, 1)(x)\right) 
\frac{\phi(y)\chi_E(y)}{(\Gamma, 1)(V_y)}\ d(\Gamma, 1)(y) \\
&=& \int \phi(y) \chi_E(y) \ d(\Gamma, 1)(y) = (\Gamma, \phi)(E).
\end{eqnarray*}

\begin{description}
 \item[Step 2]
\end{description}
The densities $\tilde{\phi}_x$ have an inconvenient defect. They are not normalised, i.e. we do not have 
\[\int_{W_x} \tilde{\phi}_x \ d(\Gamma, 1) = 1
\]
in general. We therefore write $\rho(x) = \int_{W_x} \tilde{\phi}_x d(\Gamma, 1)$ and consider the normalised densities
\[\hat{\phi}_x = \frac{\tilde{\phi}_x}{\rho(x)},\]
so that indeed $\int_{W_x} \hat{\phi}_x d(\Gamma, 1) = 1$ for every $x$ in $\Gamma$. Moreover,
\begin{equation}\label{phihat}
\int_{\Gamma} \left( \int_{W_x\cap E} \hat{\phi}_x \ d(\Gamma, 1) \right) \rho(x) d(\Gamma, 1)(x)
 = (\Gamma, \phi)(E)
\end{equation}
for every Borel set $E \subset M$. It follows from (\ref{phihat}) that $\int_{\Gamma} \rho \ d(\Gamma, 1) = 1$.

\begin{description}
 \item[Step 3]
\end{description}

In order to complete the construction of the densities $\phi_x$, we must transfer the $\hat{\phi}_x$ from $W_x$ to $\Gamma_x$. Let 
\[J_x(y) = \frac{\vert \Gamma_x \vert}{\vert \Gamma \vert } \vert \det Df_{\vert T_y \Gamma}^n \vert.\] 
That is, $J_x$ is the Jacobian of $f^n$ from $W_x$ to $\Gamma_x$ with respect to the measures $(\Gamma, 1)$ and $(\Gamma_x, 1)$. We define $\phi_x: \Gamma_x \rightarrow \mathbb{R}$ by 
\[ \phi_x(f^n(y)) = \frac{\hat{\phi}_x(y)}{J_x(y)}\]
for every $y \in W_x$.
One may readily check that 
\[\int (\Gamma_x, \phi_x) \ d(\Gamma, \rho)(x) = f_*^n (\Gamma, \phi).\] 
Indeed, given $E \subset M$, we calculate
\begin{eqnarray*}
 &&\int_{\Gamma} (\Gamma_x, \phi_x)(E) \ d(\Gamma, \rho)(x) \\
&=&\int_{\Gamma} \left( \int_{\Gamma_x} \phi_x \chi_E \ d(\Gamma_x, 1) \right) d(\Gamma, \rho)(x) \\
&=&\int_{\Gamma} \left( \int_{W_x} \phi_x \chi_{f^{-n}(E)} J_x \ d(\Gamma_x, 1)\right) d(\Gamma, \rho)(x) \\
&=& \int_{\Gamma} \left( 
\int_{W_x}\hat{\phi}_x\chi_{f^{-n}(E)}  \ d(\Gamma, 1)\right) d(\Gamma, \rho)(x) \\
&=&(\Gamma, \phi)(f^{-n}(E)) = f_*^n (\Gamma, \phi)(E).
\end{eqnarray*}

Next lemma proves that all densities $\phi_x$ in the above construction are bounded away from zero and infinity by uniform constants, i.e. independent of $x \in \Gamma$ and, more importantly, independent of the iterate $n \geq 0$.

\begin{lemma}\label{Lip1}
There exists $C>0$, independent of $n$, such that if $\phi$ satisfies $D^{-1} \leq \phi \leq D$, then each $\phi_x$ satisfies $ (D^2 C)^{-1} \leq \phi_x \leq D^2 C$. The number $C$ is uniform on a $C^2$-neighbourhood of $f$.
\end{lemma}

The proof of Lemma \ref{Lip1} is based on a simple estimate. We use the notation $B_{r}^{\Gamma,n}(x)$ to denote a $d^{\Gamma,n}$-ball in $\Gamma$, centred at $x$.

\begin{sublem}\label{estimate}
 We have 
\[B_{R_a(x)/2}^{\Gamma,n}(x) \subset V_x \subset B_{3 R_a(x)}^{\Gamma,n}(x)\]
for every $x$ in $\Gamma$.
\end{sublem}

\begin{proof}
We prove the first inclusion. The latter is very similar.
For a point $z \in \Gamma$ \emph{not} to be in $V_y$, it must satisfy
\[
\begin{cases}
d^{\Gamma,n}(z, y) \geq (1-C(r_0)) R_a(z) \\
R_a(z) \geq R_a(y) - \frac{1}{2} d^{\Gamma,n}(y, z)
\end{cases}
\]
for some small $C(r_0)>0$ that can be choosen arbitrarily close to zero upon reducing $r_0$.
Hence $d^{\Gamma,n}(y, z) \geq (1-C(r_0)) \frac{2}{3} R_a (y)$,  so $V_y$
contains a ball of $d^{\Gamma,n}$-radius 
$(1-C(r_0)) \frac{2}{3} R_a(x) > \frac{1}{2} R_a(x)$. 
\end{proof}

\begin{proof}[Proof of Lemma \ref{Lip1}.]
 Pick some $x \in \Gamma$. Recall that 
\[
 \frac{\phi_x(f^n(y))}{\phi_x(f^n(z))} = \frac{\phi(y) J(z) (\Gamma, 1)(V_z)}{\phi(z) J(y) (\Gamma, 1)(V_y)}
\]
for every $y, z \in W_x$. We shall use $\inf$ and $\sup$ as shorthand notations of $\ess \inf$ and $\ess \sup$. Thus we estimate 
\[
\frac{\sup \phi_x}{\inf \phi_x} \leq  \frac{\sup_{y \in W_x} \phi(y)}{\inf_{z\in W_x} \phi(z)} \frac{\sup_{z \in W_x} J(z)}{\inf_{y \in W_x} J(y)} 
\frac{\sup_{z \in W_x} (\Gamma, 1)(V_z)}{ \inf_{y \in W_x} (\Gamma, 1)(V_y)}.
\]
By hypothesis $D^{-1} \leq \phi \leq D$ so that 
\[\frac{\sup_{y \in W_x} \phi(y)}{\inf_{z\in W_x} \phi(z)} \leq D^2.\]
We also know from the theory of expanding maps that there is some $C_0$ such that 
\[
 \frac{J_x (z)}{J_x (y)} \leq e^{C_0 d^{\Gamma,n} (y, z)}.
\]
Indeed, taking $C_0 = \sum_{k=0}^{\infty} \tau^{k-n_0} \Lip(\log \vert \det Df_{\vert \Gamma}\vert)$ will do just fine, but it is wise to exagerate the value a bit so that it is holds on a neighbourhood of $f$. Therefore,
\[
 \frac{\sup_{z \in W_x} J(z)}{\inf_{y \in W_x} J(y)} \leq e^{C_0 3 R_a(x)}.
\]
Finally, it follows from the curvature bounds in Proposition \ref{curvature} that there exists $C_1>1$ (that can be chosen arbitrarily close to $1$ upon reducing $r_0$) such that 
\[
  C_1^{-1} \vol (\mathbb{B}^u) r^u 
\leq (\Gamma, 1)(B_{r}^{\Gamma,n} (x)) \leq C_1 \vol(\mathbb{B}^u) r^u
\]
whenever $r < d^{\Gamma,n}(x , \partial \Gamma)$.
Here $\vol(\mathbb{B}^u)$ is the volume of the unit ball in $u$-dimensional Euclidean space.
Since $W_x$ is contained in a ball of $d^{\Gamma,n}$-radius slightly larger than $R_a(x)$, say $W_x \subset B_{3R_a(x)/2}^{\Gamma,n}(x)$, it follows from Sublemma \ref{estimate} and the fact that $R_a$ is $\frac{1}{2}$-Lipschitz that 
\[\frac{\sup_{z \in W_x} (\Gamma, 1)(V_z)}{ \inf_{y \in W_x} (\Gamma, 1)(V_y)} 
\leq \frac{C_1 \vol(\mathbb{B}^u) 3^u (R_a(x)+\frac{3}{4} R_a(x))^u}{C_1^{-1} \vol(\mathbb{B}^u) 2^{-u} (R_a(x)-\frac{3}{4} R_a(x))^u} = 42^u C_1^2 \]
Thus, taking $C = e^{3 a C_0} C_1^2 42^u$, we arrive at 
\[\frac{\sup \phi_x}{\inf \phi_x} \leq D^2 C.\]
Clearly $\int \phi \ d(\Gamma, 1) = 1$ implies that $\inf \phi_x \leq 1 \leq \sup \phi_x$, and hence 
\[(D^2 C)^{-1} \leq \phi_x \leq D^2 C.\]  
\end{proof}

Consider the map
\begin{eqnarray}\label{xi}
\xi: h(r) & \rightarrow & \mathbf{A} \\ \nonumber
        x & \mapsto     & (f^n(x), R_a(x), h_x, \psi_x, \phi_x).
\end{eqnarray}
Our calculations show that $\xi_*(\Gamma, \rho)$ is indeed a lift of $f_*^n (\Gamma, \phi)$.
In other words, the push-forward under $f$ of any simple admissible measure is again admissible. 
Of course we may do the same things for admissible measures in general, simply by applying the machinery on every $(\Gamma, \phi) \in \A$; thus obtaining densities, say $\xi_{(\Gamma, \phi, f^n, a)}$ and $\rho_{(\Gamma, \phi, f^n, a)}$. We introduce the operator
\begin{eqnarray*}
\Xi_{(f^n, a)} : \M(\mathbf{A}) &\rightarrow& \M(\mathbf{A}) \\
\underline{\mu} & \mapsto & \int (\xi_{(\Gamma, \phi, f^n, a)})_* 
(\Gamma, \rho_{(\Gamma, \phi, f^n, a)}) d\underline{\mu}(\Gamma, \phi).
\end{eqnarray*}
It has the property that if $\uiota(\umu)$ is a lift of $\mu$, then 
$\uiota (\Xi_{(f^n, a)} \umu)$ is a lift of $f_*^n \mu$. That is,
\begin{equation*}
\begin{CD}
\M(\mathbf{A}) @ >  \displaystyle \Xi_{(f^n, a)} >>  \M(\mathbf{A}) \\
@V \uiota VV       @V \uiota VV \\
\M(M) @ > f_*^n  >> \M(M)
\end{CD}
\end{equation*}
commutes. In particular, $\AM$ is invariant under every neighbour of $f$, 
so Propostion \ref{AM invariant} is proved. Notice that $\{\Xi_{(f^n, a)}: n \geq 0\}$ is not a semi-group. Indeed, we do not wish to consider iterates $\Xi_{(f, a)} \circ \ldots \circ \Xi_{(f, a)}$, for doing that leaves us with little control regarding the radius of the simple admissible measures over which the lifts are distributed. On the contrary, the sequence $\Xi_{(f^n, a)} \umu$ always satisfies
$\lim_{n \rightarrow \infty} \Xi_{(f^n, a)} \umu (\A(a)) = 1$ for every $\umu \in \M(\A)$. This is because, given any simple admissible measure $(\Gamma, \phi)$, the set of points $x \in \Gamma$ for which $\Gamma_x$ has radius $a$ grows in $(\Gamma, \phi)$-measure towards $1$ as $n$ is increased.

We identify $\M(M)$ with corresponding linear continuous functionals on $C^0(M, \mathbb{R})$ in the usual manner by $\mu(\varphi) = \int \varphi d\mu$. We may consider $C^0(M, \mathbb{R})^*$ with its strong (or norm) topology, namely
\[\|\mu\|_s := \sup_{\substack{\varphi \in C^0(M, \mathbb{R}) \\ \|\varphi\|_{C^0} \leq 1}} \mu(\varphi).
\]
We may extend $f_*$ to the whole of $C^0(M, \mathbb{R})^*$ by 
\[f_* \mu (\varphi) = \mu(\varphi \circ f) \quad \forall \varphi \in C^0(M, \mathbb{R}).\]
Thus we have 
\[\|f_* \mu \| = \sup_{\substack{ \varphi \in C^0 \\ \|\varphi \|_{C^0} \leq 1}} \int \varphi \circ f \ d\mu \leq \sup_{ \substack{ \varphi \in C^0 \\ \| \varphi \|_{C^0}} \leq 1 } \int \varphi \ d\mu = \| \mu \|, \]
so that $\|f_* \| \leq 1$. (In fact $ \| f_* \| = 1$ as the equality $\|f_* \mu \| = \| \mu \|$ clearly holds whenever $\mu$ is a positive measure.)

Let $\A_1= \bigcup_{a \in I} \A(a, 1)$.
\begin{prop}\label{suppA1}
 Every admissible measure is strongly approximated by measures having a lift in $\A_1$.
\end{prop}

\begin{proof}
 A brief outline will suffice. Consider a simple admissible measure $(\Gamma, \phi)$ in $\A(a, C)$. 
We may approximate $\phi$ in the $L^1$ sense by a simple function $\sum_{i=1}^n a_i \chi_{A_i}$.
Clearly, the normalised restriction $(\Gamma, 1)_{\vert A_i}$ is strongly approximated by some admissible measure $\mu_i$ with lift in $\A_1$. Thus $\mu_{(\Gamma, \phi)} = \sum_i a_i (\Gamma, 1)(A_i) \mu_i$ is a strong approximation of $(\Gamma, \phi)$. The procedure can be done simultaneously for every simple admissible measure and hence works for admissible measures in general.
\end{proof}

We state a non-invertible analogue of Theorem 3 in \cite{PeSi}, proving existence of Gibbs-$u$ States for partially hyperbolic diffeomorphisms.

\begin{prop}\label{likePeSi}
There is a constant $C>0$, uniform in a $C^2$ neighbourhood of $f$, such that if $\mu_0$ is any admissible measure and
$\mu$ a weak accumulation point of $\frac{1}{n} \sum_{k=0}^{n-1} f_*
\mu_0$, then $\mu \in \AM (a, C)$ for every $0<a< r_0$.
\end{prop}

\begin{proof}
Take some sequence $\nu_i$ of admissible measures with lifts in $\A_1$, converging strongly to $\mu_0$. By Lemma \ref{Lip1}, there is some large number $C$ such that every weak accumulation point of $\frac{1}{n} \sum_{k=0}^{n-1} f_* \nu_i$ belongs to $\AM (a, C)$. Since $f_*$ is a strong contraction on $\M(M)$, it follows by compactness of $\AM(a, C)$ that weak accumulation points of $\frac{1}{n} \sum_{k=0}^{n-1} f_*
\mu_0$ also belong to $\AM(a, C)$.
\end{proof}

Let $\AM_f$ be the space of $f$-invariant admissible measures. Proposition \ref{likePeSi} gives this immediate corollary:

\begin{cor}\label{compactness}
$\AM_f$ is compact.
\end{cor}
\begin{proof} Proposition \ref{likePeSi} implies that there exist $a, C>0$ such that $\AM_f = \AM(a, C) \cap \M_f(M)$. Hence by Proposition \ref{admismeascomp}, 
$\AM_f$ is compact.
\end{proof}

\begin{ex}
 Let $X\subset U$ be some Borel set of positive Lebesgue measure, e.g. the basin of a physical measure. 
Consider the normalised restriction 
$m_{\vert X}$ of $m$ to $X$. It is not necessarily admissible, but may be strongly approximated by admissible measures, say $\mu_i \rightarrow m_{\vert X}$. 
From Proposition \ref{likePeSi} we know that any accumulation point $\mu_i^{\infty}$ of 
$\frac{1}{n}\sum_{k=0}^{n-1} f_*^k \mu_i$ belongs to $\AM_f$. Since $f_*$ acts as a contraction on $C^0(M, \mathbb{R})^*$ when considered with the strong topology, it follows by compactness of $\AM_f$ that any accumulation point of $\frac{1}{n}\sum_{k=0}^{n-1}f_*^k m_{\vert X}$ is also in $\AM_f$. 
\end{ex}

We take a further look at the possible uses of Proposition \ref{likePeSi}.
Let $S$ denote the set of pairs $(f, \mu)$  in $\PH \times \M(M)$ such that $\mu$ belongs to $\AM_f$.

\begin{prop}\label{closed}
$S$ is a closed subset of $ \PH \times \M(M)$.
\end{prop}

\begin{proof}
Consider a sequence $f_i$, converging to $f$ in $\PH$ in the $C^2$ topology and let $\mu_i$ be any sequence of probabilities such that $\mu_i \in \AM_{f_i}$ for every $i$. Taking a subsequence, if necessary, we may suppose that $\mu_i$ converge to some $\mu$, which is necessarily an invariant measure for $f$. Fix some small $a$ and large $C$. By Proposition \ref{likePeSi}, every $\mu_i$ belongs to $\AM(a, C)$, which is compact by Proposition \ref{admismeascomp}. Hence $\mu \in \AM(a, C)$ and we are done.
\end{proof}

\subsubsection{Ergodic admissible measures}\label{ergodic}

Inspired by Lemma 3.14 in \cite{T}, we prove here that every invariant admissible measure decomposes into ergodic admissible measures. We recall Choquet's theorem on integral representation in locally convex spaces.

\begin{thm}[Choquet \cite{Ph}]
 Let $Y$ be a locally convex topological vector space and $X$ a compact convex metrisable subset. 
Denote by $\operatorname{ex} X$ the set of extreme points of $X$. Then, given any point $p$ in $X$, 
there exists a Borel probability $\mu$ on $X$ such that
\begin{enumerate}
 \item $\mu(\operatorname{ex} X) = 1$,
\item $\ell(p) = \int \ell(x) \ d\mu(x)$ for every linear continuous $\ell: Y \rightarrow \mathbb{R}$. 
\end{enumerate}

\end{thm}

Choquet's theorem is often used to prove the ergodic decomposition theorem. Indeed, taking $Y$ to be 
$C(M, \mathbb{R})^*$, endowed with the weak* topology and  $X = \M(M)$, 
one obtains

\begin{cor}
 Given any $f$-invariant probability $\mu$, there exists a Borel probability $\hat{\mu}$ on $\M_f(M)$ such that 
\begin{enumerate}
 \item $\hat{\mu}(\M_f^{\text{erg}}(M)) = 1$, 
\item $ \mu = \int_{\M_f^{\text{erg}}(M)} \nu \ d\hat{\mu}(\nu)$.
\end{enumerate}
\end{cor}

One may check that $\hat{\mu}$ is unique and given by 
\begin{equation}\label{muhat}
\hat{\mu}(E) = \mu \left( \ \bigcup_{\nu \in E} \B(\nu) \ \right)
 \end{equation}
for every Borel set $E \subset \M(M)$.
The purpose os this section is to prove an analogous result about invariant admissible measures.
\begin{prop}\label{ergdecomp}
Let $\mu$ be an $f$-invariant admissible measure. Then there exists a unique Borel probability $\hat{\mu}$ on $\AM_f$ such that 
\begin{enumerate}
 \item $\hat{\mu}(\AM_f^{\text{erg}}) = 1$ and
\item $\mu = \int_{\AM_f^{\text{erg}}} \nu \ d\hat{\mu}(\nu)$.
\end{enumerate}
\end{prop}

Again, $\hat{\mu}$ must be given by (\ref{muhat}), so it is unique. The nontrivial part of the statement is that $\hat{\mu}$ thus defined satisfies $\hat{\mu}(\AM_f^{\text{erg}}) = 1$. Proposition \ref{ergdecomp} is an immediate corollary of Choquet's theorem and the following characterisation of ergodic admissible measures.

\begin{lemma}[Cf. proof of Lemma 3.14 in \cite{T}]
 The set of extreme points of $\AM_f$ is precisely the set $\AM_f^{\text{erg}}$ 
of ergodic admissible measures.
\end{lemma}

\begin{proof}
It is clear that if $\mu$ belongs to $\AM_f^{\text{erg}}$, then it is an extreme point of $\AM_f(M)$. Indeed, since $\mu$ cannot be written as a nontrivial convex combination of distinct measures in $\M_f(M)$, it certainly cannot be written as a nontrivial convex combination of distinct measures in $\AM_f$.
Conversely, suppose that $\mu$ in $\AM_f$ is not ergodic, say $f^{-1}(E) = E$ and $0< \mu(E)<1$.
Choose some lift $\umu$ of $\mu$ and fix $\epsilon>0$ arbitrarily. Given any simple admissible measure 
$(\Gamma, \phi)$ such that $(\Gamma, \phi)(E)>0$, we may find an admissible measure $\nu_{(\Gamma, \phi)}$ which $\epsilon$-approximates $(\Gamma, \phi)_{\vert E}$ in the strong topology. Let $\unu_{(\Gamma, \phi)}$ be lifts of such. Then
\[\nu := \int \unu_{(\Gamma, \phi)} d\umu(\Gamma, \phi) \]
is a strong $\epsilon$-approximation of $\mu_{\vert E}$.

Any accumulation point $\nu_{\infty}$ of the sequence $\frac{1}{n} \sum_{k=0}^{n-1} f_*^k \nu$ is admissible. Furthermore, since $\mu_{\vert E}$ is $f$-invariant we have that 
$\|\nu_{\infty} - \mu_{\vert E} \|_s \leq \epsilon$. As $\epsilon$ is arbitrarily small, it follows that
\[ \inf_{\nu \in \AM_f} \| \nu - \mu_{\vert E} \|_s = 0 \]
so, by compactness of $\AM_f$, $\mu_{\vert E}$ is admissible.
\end{proof}

\subsubsection{Generic Carriers} \label{generic}

Given any $\mu\in \M(M)$ we say that a point $x$ is generic for $\mu$ if it is contained in the basin $\B(\mu)$. Similarily, a carrier $\Gamma$ is $\mu$-generic if $(\Gamma, 1)$-almost every point 
$x \in \Gamma$ is generic for $\mu$. Let $\mu$ be an admissible measure. Then by definition, there exists $\umu \in \M(\A)$ such that 
\begin{equation}\label{integralrep}
 \mu(\B(\mu)) = \int (\Gamma, \phi)(\B(\mu)) \ d\umu(\Gamma, \phi).
\end{equation}
Now suppose $\mu$ is $f$-invariant ergodic, so that $\mu(\B(\mu)) = 1$. Then the representation (\ref{integralrep}) implies that $(\Gamma, \phi) (\B(\mu))= 1$ for $\umu$-almost every $(\Gamma, \phi) \in \A$. In particular, $\mu$-generic carriers exist.

\subsection{Stable manifolds}

We briefly review Pesin's work \cite{Pe} on stable manifolds, an indispensible tool in smooth ergodic theory. He proves almost everywhere existence of stable manifolds with respect to any hyperbolic invariant measure for $C^{1+\alpha}$ maps ($\alpha >1$). Later it has been remarked that having stable manifolds is a pointwise property, depending on non-uniform hyperbolicity along a given orbit, i.e. one does not have to mention any invariant measure in order to talk about stable manifolds. Almost everywhere existence is a consequence of Oseledet's Theorem \cite{O}. Moreover, stable manifolds may be constructed for quite general sequences of diffeomorphisms (see \cite{BP}); in particular the theory works fine for local diffeomorphisms. Here we state a weak form of Pesin's theorem which, nevertheless, is quite sufficient for our needs.
We shall write $\mathcal{N}$ for the set of points $x$ in $U$ for which $\lambda_+^c(x) < 0$.

\begin{thm}\label{stablemfds}
 There exists a measurable function 
$r:\mathcal{N} \rightarrow (0, \infty)$ satisfying 
\begin{equation}\label{subexponential}
\lim_{n \rightarrow \infty} \frac{1}{n} \log r(f^n(x)) = 0, 
\end{equation}
and $C^1$ maps $\Psi_x : E^c_x (r(x)) \rightarrow E^u_x$, such that
 the submanifolds 
\[W_{loc}(f; x) := \exp_x \Graph \Psi_x\] satisfy:
\begin{enumerate}
\item $d(f^n(x), f^n(y))\rightarrow 0$ exponentially fast as $ n \rightarrow \infty$ for every $y \in W_{loc}(f; x)$,
 \item $T_y W_{loc} (f; x) = E^c(y)$ at any $y\in W_{loc} (f; x), \ x \in \mathcal{N}$,
\item $f(W_{loc}(f; x)) \subset W_{loc}(f; f(x))$.
\end{enumerate}
\end{thm}

\subsection{Absolute continuity}

We fix a function $r: \mathcal{N} \rightarrow (0, \infty)$ as in Theorem \ref{stablemfds} so that we obtain a family 
$\mathcal{W} = \{W_{loc}(x): x \in \mathcal{N}\}$ of local stable leaves. Given two carriers $\Gamma_1, \Gamma_2$, let $\mathcal{D}(h_{\Gamma_1, \Gamma_2})$ be the domain $\{p \in \Gamma_1: W_{loc}(p)\cap \Gamma_2 \neq \emptyset \} \subset \Gamma_1$. It is understood that 
$r$ is small enough so that every carrier intersect every local 
stable manifold in \emph{at most} one point. Thus we may define the holonomy map
\begin{eqnarray*}
 h_{\Gamma_1, \Gamma_2} : \mathcal{D}(h_{\Gamma_1, \Gamma_2}) &\rightarrow& \Gamma_2 \\
p &\mapsto&  W_{loc}(p) \cap \Gamma_2.
\end{eqnarray*}
It is clear that 
\[(\Gamma_1, 1)(\mathcal{D}(h_{\Gamma_1, \Gamma_2})) \rightarrow 1 \textrm{ as } \Gamma_2 \overset{\mathbf{K}}{\rightarrow} \Gamma_1.\] 
Since the local stable manifolds are open discs, the condition $W_{loc}(p)\cap \Gamma_2 \neq \emptyset$ is robust under small perturbations of $\Gamma_2$. Consequently, the map 
$\Gamma_2 \mapsto (\Gamma_1, 1)(\mathcal{D}(h_{\Gamma_1, \Gamma_2}))$ is lower semi-continuous.

Let $\mu$ be the restriction of $(\Gamma_1, 1)$ to $\mathcal{D}(h_{\Gamma_1, \Gamma_2})$ and $\nu$ the restriction of $(\Gamma_2, 1)$ to $h_{\Gamma_1, \Gamma_2}(\mathcal{D}(\Gamma_1, \Gamma_2))$ (but not normalised). We define the Jacobian of $h_{\Gamma_1, \Gamma_2}$ by the Radon-Nikodym derivative
\[\Jac(h_{\Gamma_1, \Gamma_2}) = \frac{d(h_{\Gamma_2, \Gamma_2})_*^{-1} \nu}{d \mu}.\]

\begin{thm}[Absolute Continuity \cite{P, BP}] \ 

 \begin{enumerate}
  \item All holonomy maps are absolutely continuous, i.e. $h_{\Gamma_1, \Gamma_2}$ sends zero 
$(\Gamma_1, 1)$-measure sets into zero $(\Gamma_2, 1)$-measure sets.
\item There is a uniform constant $C>0$ such that 
\[\vert \Jac(h_{\Gamma_1, \Gamma_2}) - 1\vert \leq C d^{\K}(\Gamma_1, \Gamma_2).\]
 \end{enumerate}
\end{thm}

Let $\mathcal{F}$ be any measurable union of local stable manifolds, e.g. $\mathcal{F} = \mathcal{N}$, or $\mathcal{F} = \mathcal{N} \cap \B(\mu)$ for some physical measure $\mu$.

\begin{cor}
 The map
\[ \A \ni (\Gamma, \phi) \mapsto (\Gamma, \phi)(\F) \in \mathbb{R} \]
 is lower semi-continuous.
\end{cor}

It is a general fact that if $\varphi:X\rightarrow \mathbb{R}$ is a lower (upper) semi-continuous function on some probability space $X$, so is $\M(X) \ni \mu \mapsto \int \varphi d\mu \in \mathbb{R}$. Applied to the current context this becomes:

\begin{cor}
 The map 
\[\M(\A) \ni \umu \mapsto \uiota(\umu)(\F) \in \mathbb{R} \label{lsc2} \]
is lower semi-continuous.
\end{cor}

%-----------------------------------------------------------

\section{Finitude of physical measures and the no holes property}

Having developed the necessary tools, we are now ready to begin the proof of Theorem A. Although it may appear rather different, our proof resembles that of \cite{BV} in spirit, simply replacing Gibbs-$u$ states with admissible measures. Still there is one profound difference: we do not employ the technique of Lebesgue density points when proving the no holes property. 

\begin{proof}[Proof of Theorem A]

Recall the statement: Every mostly contracting system has a finite number of physical measures and the union of their basins of attraction cover Lebesgue almost every point in the trapping region $U$. We have seen in section \ref{ergodic} that every $f$ in $MC$ has some ergodic admissible measure. 
The proof has three phases:
\begin{enumerate}
 \item Every ergodic admissible measure is also a physical measure.
\item There are finitely many ergodic measures, say $\AM_f^{\text{erg}} = \{\mu_1, \ldots, \mu_N\}$.
\item The combined basin $\B(\mu_1) \cup \ldots \cup \B(\mu_N)$ has full Lebesgue measure in $U$. In particular, there is no room for yet another physical measure.
\end{enumerate}

Let $\mu$ be any ergodic admissible measure and $\umu$ a lift of $\mu$. Recall that we denote by $\mathcal{N}$ the set of points in $U$ whose maximum central Lyapunov exponent are negative. The mostly contracting hypothesis implies that $(\Gamma, \phi)(\mathcal{N}) > 0$ for every simple admissible measure $(\Gamma, \phi)$. Hence
\[\mu(\mathcal{N}) = \int_{\A} (\Gamma, \phi)(\mathcal{N}) d\umu(\Gamma,\phi) >0.\]
The set $\mathcal{N}$ is $f$-invariat by definition, so it follows by ergodicity of $\mu$ that 
\[\mu(\mathcal{N}) = \int_{\A}(\Gamma, \phi)(\mathcal{N})d\umu(\Gamma, \phi) = 1.\] 
Thus $\umu$-almost every simple admissible measure satisfies $(\Gamma, \phi)(\mathcal{N}) = 1$. We have seen in Section \ref{generic} that $\umu$-almost every $(\Gamma, \phi)$ is $\mu$-generic. In particular, there exists some carrier $\Gamma$ such that $(\Gamma, 1)(\B(\mu)\cap \mathcal{N}) = 1$. It follows from absolute continuity that 
\[A:=\bigcup_{x \in \B(\mu)\cap \mathcal{N}} W_{loc}(f; x)\] 
has positive Lebesgue measure. As $A$ is a subset of $\B(\mu)$, it follows that $\mu$ is a physical measure.

Next we show that $\AM_f^{\text{erg}}$ is finite. Since every ergodic admissible measure is a physical measure, there can be at most a countable number of them. If there were to be infinitely many ergodic admissible measures, say $\mu_1, \mu_2, \ldots$ then there would exist some sequence $\mu_{n_j}$ of distinct physical measures converging to some measure $\mu$. By compactness of $\AM_f$, $\mu$ must be admissible. Indeed, if $\umu_{n_j}$ are lifts of the $\mu_{n_j}$ in some $\M(\A(a, C))$ (see Proposition \ref{likePeSi}), then any accumulation point $\umu$ of $\umu_{n_j}$ is a lift of $\mu$. 

Writing $\alpha_n = \mu(\B(\mu_n))$, the ergodic decompositon of $\mu$ takes the form
$\mu = \sum_{n=0}^{\infty} \alpha_n \mu_n$. Now let $\B_n = \B(\mu_n) \cap \mathcal{N}$ for every $n \geq1$. Since $\mathcal{N}$ has full measure with respect to any $f$-invariant probability, we have $\mu_n (\B_n)=1$ and $\alpha_n = \mu(\B_n)$ for every $n$. Pick one $k$ such that $\alpha_k>0$. Using  Corollary \ref{lsc2} we obtain
\[\liminf_{j \rightarrow \infty} \uiota (\umu_{n_j})(\B_k) \geq \uiota (\umu)(\B_k) = \mu(\B_k) = \alpha_k>0.\]
But this is absurd since $\mu_{n_j}(\B_k)=0$ unless $n_j = k$, which can certainly be true for at most one value of $j$.

Let $\AM_f^{\text{erg}}=\{\mu_1,\ldots, \mu_N\}$. To complete the proof of Theorem A it remains to prove that these are the only physical measures supported in the trapping region $U$, and that their combined basin 
$\mathcal{B}(\mu_1)\cup \ldots \cup \mathcal{B}(\mu_N)$
has full Lebesgue
measure in $U$. But of course the former follows from the latter. Fix therefore some small $\epsilon >0$ and pick some $\nu_0 \in \AM$ with $\| \nu_0 - m_{\vert U} \|_s < \epsilon$. Here, $m_{\vert U}$ denotes the normalised restriction of Lebesgue measure to the trapping region. Let $\unu_0$ be any lift of $\nu_0$. We denote by $\nu_n$ the averaged sums of pushforwards of $\nu_0$, and $\unu_n$ their lifts given by the construction in section \ref{technique}:

\begin{equation*}
\begin{CD}
\M(\mathbf{A}) \ni \unu_0 @ >  \frac{1}{n}\sum_{k=0}^{n-1}\Xi_{(f^k, a)} >> \unu_n \in \M(\mathbf{A}) \\
@V \uiota VV       @VV \uiota V \\
\M(M)\ni \nu_0 @ >  \frac{1}{n} \sum_{k=0}^{n-1} f_*^k  >> \nu_n \in \M(M)
\end{CD}
\end{equation*}

Let $\nu$ be an accumulation point of $\nu_n$. Then there is some subsequence $\unu_{n_j}$ of $\unu_n$, converging to a lift $\unu$ of $\nu$. Since $\nu$ is admissible, it has an ergodic composition of the form
\[
\nu = \alpha_1 \mu_1 + \ldots + \alpha_N \mu_N. \\
\]
By ergodicity, $\mu_i (\B(\mu_i)) = 1$ for each $1\leq i \leq N$. Hence
\[\nu(\B(\mu_1)\cup \ldots \cup \B(\mu_N)) = 1.\]
We have already seen that $\mu(\mathcal{N}) = 1$ for every ergodic admissible measure. Hence $\nu(\mathcal{N})=1$ as well. Let $\F = (\B(\mu_1)\cup \ldots \cup \B(\mu_N)) \cap \mathcal{N}$. 
Since $\F$ is a union of local stable manifolds, it follows from Corollary \ref{lsc2} that 
\[\liminf_{j \rightarrow \infty} \nu_{n_j}(\F)
= \liminf_{j \rightarrow \infty} \uiota(\unu_{n_j})(\F)
\geq \uiota(\unu)(\F) = \nu(\F) = 1.
\]

By invariance of $\F$, $\nu_0(\F)=1$, so that  
\[m_{\vert U} (\B(\mu_1) \cup \ldots \cup \B(\mu_N)) > 1-\epsilon\] as required.
\end{proof}

\section{Robustness and statistical stability}

The goal of this section is to prove Theorem B. The first part (openness of $MC$) is obtained through a characterisation of the mostly contracting hypothesis in terms of negative integrated central Lyapunov exponents for invariant admissible measures. The remaining part of Theorem B requires some estimates on the sizes of stable manifolds, and will be dealt with separately.

\subsection{A characterisation of the mostly contracting hypothesis}\label{character}

The definition of maximum central Lyapunov exponent given in section \ref{mostcont} is naturally modified to take arguments in the space of invariant measures. Recall the set $S = \{(f, \mu) \in \PH \times \M(M): \mu \in \AM_f \}$.

\begin{defn}
 The integrated maximum central Lyapunov exponent is the map
\begin{eqnarray*}
 \hat{\lambda}_+^c: S & \rightarrow & \mathbb{R} \\
(f, \mu) & \mapsto & \int \lambda_+^c d\mu.
\end{eqnarray*}
\end{defn}

\begin{prop}\label{alternative}
A partially hyperbolic system $f$ is
mostly contracting along the central direction if and only if 
the integrated maximum central Lyapunov exponent is negative on any admissible invariant measure.
\end{prop}

\begin{proof}
The `only if' was implicitly dealt with in the proof of Theorem A. Indeed, given any admissible $\mu$, 
we may write 
\[\mu(\mathcal{N}) = \int (\Gamma, \phi)(\mathcal{N})\ d\underline{\mu}(\Gamma, \phi)\] 
for some lift $\umu$ of $\mu$. Under the mostly contracting hypothesis we have
$(\Gamma, \phi)(\mathcal{N})>0$ for every $(\Gamma, \phi)\in \mathbf{A}$, so $\mu(\mathcal{N}) >0$. Now, $\mathcal{N}$ is an $f$-invariant set, so if $\mu$ is ergodic, then $\mu(\mathcal{N}) = 1$. If not, it decomposes into ergodic admissible measures, so 
\[
 \mu(\mathcal{N}) = \int_{\AM_f^{\text{erg}}} \nu(\mathcal{N}) \ d\hat{\mu}(\nu) = 1.
\]
Hence $\hat{\lambda}_+^c(\mu) = \int \lambda_+^c d\mu = \int_{\mathcal{N}} \lambda_+^c d\mu < 0$ as required.

To prove the converse, choose an arbitrary $C^{1+\Lip}$ disc $D \subset 
U$, transversial to $E^c$. Given any point $p\in D$, there is some $n \geq 0$ such that $f^n(D)$ is tangent to $S^u$ at $f^n(p)$. Moreover, provided that $n$ is large enough, there is some neighbbourhood $N$ of $p$ such that $N$ is an admissible manifold. In particular, $N$ contains some carrier $\Gamma$. By invariance of $\mathcal{N}$, it suffices to show that $(\Gamma, 1)(\mathcal{N})>0$. 

Let $\unu_0 = \delta_{(\Gamma, 1)}$ and for $n \geq 1$ define 
\[\unu_n = \frac{1}{n} \sum_{k=0}^{n-1} \Xi_{(f^k, a)} \unu_0, \quad \nu_n = \uiota (\unu_n).\]
Again, by invariance of $\mathcal{N}$, it suffices to show that $\nu_n(\mathcal{N})>0$ for some $n \geq 0$.
Choose some convergent subsequence $\unu_{n_j} \rightarrow \unu$ and denote 
$\uiota( \unu)$ by $\nu$. We have $\nu \in \AM_f$ so, by hypothesis, $\int \lambda_+^c d\nu <0$. Applying Corollary \ref{lsc2} yields $\iota(\unu_{n_j})(\mathcal{N})>0$ for every large value of $j$, so $\nu_{n_j}(\mathcal{N})>0$ as required.
\end{proof}

\subsection{Semi-continuity of Lyapunov exponents}

It is clear that when $E^c$ is one-dimensional, $\hat{\lambda}_+^c: S \rightarrow \mathbb{R}$ is continuous. In general this property may fail, due to interaction between several central directions. Still, it does satisfy a semi-continuity property which is well sufficient for our needs.

\begin{lemma}\label{semicont}
The integrated maximum central Lyapunov exponent 
$\lambda_+^c : S
\rightarrow \mathbb{R}$ is upper semi-continuous.
\end{lemma}

\begin{proof}
Fix $\epsilon > 0$ arbitrarily and take $N$ large so that 
\[
 \frac{1}{N} \int \log \|Df^N \vert_{E^c}\| d\mu < \lambda_+^c (f,\mu) +
\epsilon.
\]
Choose thereafter a neighbourhood $\mathcal{V}$ of $(f, \mu)$ in $S$, small
enough for
\[
 \frac{1}{N} \int \log \|Dg^N \vert_{E_g^c}\| d\mu_g < \lambda_+^c (f, \mu) +
\epsilon
\]
to hold for any pair $(g, \mu_g)\in \mathcal{V}$. We have
\begin{eqnarray*}
 \lambda_+^c(g, \mu_g) = \limsup_{n \rightarrow \infty} &=&
\frac{1}{n} \int \log \| Dg^n \vert_{E_g^c} \| d\mu_g \\
& \leq & \lim_{k \rightarrow \infty} \sum_{j=0}^{k-1}
\int \frac{1}{N} \log \|Dg^N \vert_{E_g^c} (g^{jN}(x))\| d\mu_g (x) \\
& \leq & \lambda_+^c (\mu) + \epsilon
\end{eqnarray*}
which proves the lemma.
\end{proof}

\begin{proof}[Proof of Theorem B, part 1]

Using the characterisation of the mostly contracting hypothesis given by proposition \ref{alternative}, we find that 
\[MC = \{f \in \PH: \hat{\lambda}_+^c(\mu) < 0 \quad \forall \mu \in \AM_f\}.\]
Pick some $f\in MC$. By compactness of $\AM_f$ (Proposition \ref{compactness}) and semi-continuity of $\hat{\lambda}_+^c$ (Lemma \ref{semicont}), there is a finite collection 
$\{\mathcal{U}_i \times \mathcal{V}_i \}_{i=1}^n \subset \PH \times \M(M)$ on which $\hat{\lambda}_+^c$ is negative, and such that 
$\bigcup_{i=1}^n \mathcal{U}_i \times \mathcal{V}_i \supset \AM_f$. Let $\mathcal{U} = \bigcap_{i=1}^n 
\mathcal{U}_i$. Since $S$ is closed (Proposition \ref{closed}), we have 
\[ S \cap (\mathcal{U} \times \M(M)) \subset \bigcup_{i=1}^n \mathcal{U}_i \times \mathcal{V}_i.\]
Hence $\hat{\lambda}_+^c$ is negative on $\AM_g$ for every $g\in \mathcal{U}$.
\end{proof}

\subsection{Large stable manifolds}\label{large}

The proof of the semi-continuity of the number of physical measures, as a function on $MC$ (part 2 of Theorem B), relies on certain estimates of the sizes of stable manifolds. The idea is to show that the basin of each ergodic admissible measrue is, to a certain extent, foliated by rather large stable manifolds; and that, as a consequence of this, no other ergodic admissible measure is allowed to lie very near it, lest their basins intersect.

Theorem \ref{stablemfds} announces the existence of an \emph{invariant} family of local stable manifolds associated to points in $\mathcal{N}$. However, when dealing with basins of measures, what one really cares about are the stable sets
\[W(f;x) = \{y : d(f^n(y), f^n(x)) \rightarrow 0 \text{ as } n \rightarrow \infty\}.\]
For if $x$ is in the basin of some measure $\mu$, so is the whole of $W(f;x)$. But we do not know very well how $W(f;x)$ looks in general. All we know is that if $x \in \mathcal{N}$, then $W(f;x)$ contains some small embedded disc $W_{loc}(f;x)$.

Let $K>0$ and define $L_K(f)$ to be the set of points $x \in U$ for which $W(f; x)$ contains a disc of radius $K$, centred at $x$.

\begin{lemma}\label{size}
 Suppose $f \in MC$. Then there are positive constants 
$K, \theta$, and a $C^2$-neighbourhood 
$\mathcal{U}$ of $f$ such that $\mu(L_K(g))\geq \theta $
for every $g\in \mathcal{U}$ and $\mu\in \AM_g$.
\end{lemma}

The proof of Lemma \ref{size} is a fairly direct consequence of an auxiliary result regarding the existence of a large set of points with uniformly hyperbolic behaviour. As a consequence of Lemma \ref{semicont}, we may fix some small neighbourhood $\mathcal{U}$ of $f$ and a number $\lambda<0$ such that 
\begin{equation*}
 \hat{\lambda}_+^c(g,\mu)<\lambda < 0
\end{equation*}
for every $g \in \mathcal{U}$ and $\mu \in \AM_g(M)$. We also fix some $\epsilon$ small enough that $\lambda+4\epsilon <0$, and $N$ large so that 
\[\int \frac{1}{N} \log \|D^cf^N(x)\|d\mu(x)<\lambda+\epsilon <0.\]
Let $H(g)$ be the set of points $x \in M$ such that 
\[
		\prod_{j=0}^{n-1} \|D^c g^N(g^{Nj}(x)) \| 
		\leq e^{nN(\lambda+3\epsilon)} 					
									\]
for every $n \geq 1$. 

\begin{lemma}\label{largerthantheta}
 There exists $\theta>0$ such that 
$\mu(H(g)) > \theta$ for every $g\in \mathcal{U}$ and $\mu \in \AM_g$.
\end{lemma}

The proof of Lemma \ref{largerthantheta} is a blend of Pliss' Lemma and Birkhoff's Ergodic Theorem. The former is used to achieve good hyperbolic behaviour for many points (positive frequency) along a fixed orbit. The latter transformes this positive frequency into positive measure. The idea comes from Ma\~n\'e's proof of Oseledet's theorem \cite{M}. However simple it may be, it is quite an astonishing argument. For, at a first glance, it is not even clear why $H(g)$ should be nonempty.

\begin{lemma}[Pliss' Lemma \cite{Pl}]
Let $h<A$ be real numbers and $a_0, \ldots, a_{k-1}$ some finite sequence such that $\min \{a_0, \ldots, a_{k-1} \} \geq h$ and
\[\sum_{i=0}^{k-1} a_i \leq kA. \]
For every $\epsilon>0$ there exist integers $0 \leq k_1 < \ldots< k_l < k-1$, with 
$l \geq k \frac{\epsilon}{A+\epsilon-h}$, such that
\[\sum_{j=k_i}^n a_j \leq (n-k_i)(A+\epsilon) \]
for every $1\leq i \leq l$ and $k_i \leq n \leq k-1$. 
\end{lemma}

A concise proof of Pliss Lemma can be encountered in \cite{ABV}.

\begin{proof}[Proof of Lemma \ref{largerthantheta}]
Suppose, without loss of generality, that $\mu$ is ergodic. The general case then follows from the ergodic decomposition theorem.
To simplify notation, write $\zeta(x)=\frac{1}{N} \log \|D^c g^N(x)\|$. 
By ergodicity of $\mu$, there is some point $x_0 \in M$ such that 

\begin{itemize}
 \item $\displaystyle \lim_{n\rightarrow \infty} 
 	\frac{1}{n}\sum_{k=0}^{n-1}\zeta(g^{k}(x_0)) = \int \zeta d\mu$, 
 \item $\lim_{n\rightarrow \infty} \frac{1}{n}\#\{0\leq k 
		\leq n: g^k(x_0) \in H(g)\} = \mu(H(g,N))$.
\end{itemize}

Consider the nested sequence of sets 
\[H_m(g) = \{x\in M: \prod_{j=0}^{n-1} \|D^c g^N( g^{jN}(x)) \| 
\leq e^{nN(\lambda+3\epsilon)} \quad \forall 0< n \leq m \}.\]
Clearly, $H_{m+1}(g) \subset H_m(g)$ and $H(g) = \bigcap_{m\geq 1} H_m(g)$, so if we find $\theta>0$ with $\mu(H_m(g)) \geq \theta$ for all $m \geq 1$, then we also have $\mu(H(g))\geq \theta$.

To this end, take some large multiple of $N$, say $k N$, satisfying 
\[\frac{1}{kN}\sum_{j=0}^{kN-1} \zeta(g^j(x_0)) < \int \zeta d\mu + \epsilon < \lambda + 2\epsilon.\]
We can decompose the orbit $x_0, g(x_0), \ldots, g^{k N-1}(x_0)$ into $N$ disjoint subsets, jumping $N$ iterates at each time: $g^j(x_0), g^{j+N}(x_0), \ldots, g^{j+(k-1)N}(x_0)$, $j=0, \ldots N-1$. Since the average of $\zeta$ along the $x_0, g(x_0), \ldots, g^{k N-1}$ is less than $\lambda + 2\epsilon$, so must be the case for at least one of the sub-orbits. In other words, there  is at least one $p\in \{0, \ldots, N-1\}$ satisfying 
\[\frac{1}{k} \sum_{j=0}^{k-1} \zeta(g^{jN+p}(x_0)) < \lambda + 2 \epsilon.\]
 
Let $h$ be a lower bound for $\zeta$, say 
\[ h = \inf_{g\in \mathcal{U}} \inf_{x \in M} \frac{1}{N} \log\|(D^c g^N(x))^{-1}\|.\] 
According to Pliss' Lemma, there is some $l \geq k \delta$, where $\delta = \frac{\epsilon}{\lambda+4\epsilon - h}$, and  $0\leq k_1< \ldots < k_l < k-1$, such that 
\begin{equation*}
\frac{1}{n-k_i} \sum_{j=k_i}^n \zeta(g^{jN+p}(x))\leq \lambda+3\epsilon  
\end{equation*}
for each $1\leq i \leq l$ and $k_i\leq n \leq k-1$.
Clearly $k_{l-m} < k-1-m$ for each $m\geq 1$, so $g^{k_i N+p}(x)\in H_m(g)$ for every $i < k-1-m$. Thus every orbit of length $k N$ starting at $x_0$ has at least $l-m \geq \delta k - m$ visits to $H_m(g)$. Recall that $x_0$ was chosen so that the frequency of visits to $H_m(g)$ is equal to $\mu(H_m(g))$. Therefore $\mu(H_m(g)) \geq \frac{\delta}{N}$ and the proof follows by taking $\theta = \frac{\delta}{N}$.
\end{proof}

Inspired by \cite{BDP, Ta}, we now dig into the proof of Lemma \ref{size}. The tactics of the proof is to find some $K>0$ such that $L_K(g) \supset H(g)$ for every $g\in \mathcal{U}$.

\begin{proof}[Proof of Lemma \ref{size}]
 Let $\sigma = e^{N(\lambda + 4 \epsilon)/2}$. By continuity of $D^c g$, 
there is some small $K > 0$ such that for every $x \in H(g)$ we have
\begin{equation}\label{centralcontraction}
 \| D^c g^N(y) v \| \leq \sigma \|v\|
\end{equation}
whenever $d(x, y) \leq K$ and $v \in E_y^c(g)$. Upon possibly reducing $K$ we may suppose that $g_{\vert B_{K}(x)}$ is injective at any $x\in M$ and for any $g$ in $\mathcal{U}$.
 We shall prove that if $x \in H(g)$, then $W^s(g,x)$ contains a disc of radius $K$,
 centred at $x$. 
Indeed, it follows from (\ref{subexponential}) that we can choose $j$ such that 
$K \sigma^j < r(g^{j N}(x))$. Let $D$ be the disc of radius $K \sigma^j$, centred at $g^{j N}(x)$ in $W_{loc}^s(g, g^{j N}(x))$. We claim that 
$D':=(g_{\vert B_K(x)})^{-j N}(D) \subset W^s(g, x)$ contains a disc of radius $K$, 
centred at $x$. The proof is by contradiction.

Suppose there exists $y\in \partial D'$ with $d^{D'}(x, y) < K$. Then, by (\ref{centralcontraction}), we have $d^{g^{j N}(D')}(g^{j N}(x), g^{j N}(y)) < K \sigma^j$. But this is absurd, since $g^{j N}(y) \in \partial D$ and $D$ has radius $K \sigma^j$.
\end{proof}

Let $\mathcal{L}(g) = \{(\Gamma, \phi) \in \A: (\Gamma, \phi)(L_K(g)) > \theta/2\}$.

\begin{cor}\label{manylargemanifolds}
 Let $\mathcal{U}$ be as in Lemma \ref{size} and 
$g\in \mathcal{U}$. Suppose $\mu \in \AM_g(M)$ and let
$\umu$ be any lift of $\mu$. Then $\umu(\mathcal{L}(g))>\theta/2$. 
\end{cor}

\begin{proof}
 Pick some number $\theta'<\theta$, and let $\varphi:\A \rightarrow \mathbb{R}$ 
be the function 
\[(\Gamma, \phi) \mapsto (\Gamma, \phi)(L(g)).\] Thus we have $0\leq \varphi \leq 1$ and, from Lemma \ref{size}, $\int \varphi d\umu>\theta$. Now take $\varphi' = 1-\varphi$. Again we have $0\leq \varphi' \leq 1$, but this time $\int \varphi' d\umu < 1-\theta$. Applying Chebychev's inequality we get
\[ \umu(\{ \varphi' \geq 1- \theta'\}) \leq \frac{\int \varphi' d\umu}{1-\theta'} <\frac{1-\theta}{1-\theta'}.\]
After rearranging we obtain $\umu(\{\varphi > \theta'\}) > \frac{\theta-\theta'}{1-\theta'}$, and the result follows by taking $\theta' = \theta/2$.
\end{proof}

\subsection{The natural extension and balanced lifts}\label{ILBL}

As already mentioned, lifts of admissible measures are not unique. In this section, we define the class of \emph{balanced lifts}. They are lifts with special properties that turn out to be important in the proof of part 2 and 3 of Theorem B. To get a flavour of what it means for a lift to be balanced, we cheat a bit and let the reader know that the atomic lifts considered in Example \ref{atomiclift} are not balanced, whereas that in Example \ref{nonatomiclift} is.

Given a system $f\in \PH(U, S^u)$ with attractor $\Lambda = \bigcap_{n\geq0}f^n(U)$, we associate to it the inverse limit
\[
\hat{\Lambda}_f = \{\mathbf{x}=(\ldots, x_{-2}, x_{-1}, x_0) \in \Lambda_f^{\mathbb{Z}\setminus\mathbb{N}}: f(x_{i-1})= x_i 
\forall i\leq 0\},
\] accompanied with the map 
\begin{eqnarray*}
\hat{f}: \hat{\Lambda}_f  & \rightarrow & \hat{\Lambda}_f \\
(\ldots x_{-2}, x_{-1}, x_0) & \mapsto & (\ldots x_{-1}, x_0, f(x_0)).
\end{eqnarray*}

Thus $\pi\circ\hat{f} = \pi \circ f$, where $\pi$ is the projection to the $0$th coordinate. Elements of the inverse limit are possible histories for points in $\Lambda_f$. Due to the domination property of $f$, we may assign to each such history, a unique direction in the Grassmannian through
\begin{eqnarray*}
\mathsf{g}:\hat{\Lambda}_f &\rightarrow& G^u M \\
\mathbf{x} &\mapsto& \bigcap_{i \geq0} Df^i(x) S_{x_{-i}}^u.\\
\end{eqnarray*}

Given $\mu\in \mathcal{M}(\Lambda_f)$ there exists a unique measure $\mu^{\ominus}$ in $\M(\hat{\Lambda}_f)$, invariant under $\hat{f}$. We call $\mu^{\ominus}$ the natural extension of $\mu$. 
We need some auxiliary notation in order to define the notion of balanced lifts.
Let $\Pi: \A \rightarrow G^u M$ be the projection $(r, x, h, \psi, \phi) \mapsto (x, h)$. 
Whenever $\mu$ and $\nu$ are two measures on the same measurable space and $B \geq 1$ is some constant, the notation $\mu \overset{B}{\sim} \nu$ means that $B^{-1} \nu \leq \mu \leq B \nu$.
\begin{defn}
We say that a lift $\underline{\mu}$ of $\mu \in \AM_f$ is balanced  if there is $B\geq1$ such that $\Pi_*\underline{\mu} \overset{B}{\sim} \mathsf{g}_* \mu^{\ominus}$ (in which case we say that $\umu$ is $B$-balanced).
\end{defn}

In particular, if $\umu$ is a balanced lift of $\mu$, we have $\Pi_*^M \umu \overset{B}{\sim} \mu$, where $\Pi^M: \A \rightarrow M$ is the projection $(r, x, h, \psi, \phi) \mapsto x$. The question arises as to whether such lifts are always to be found. Luckily, the disintegration technique described in Section \ref{technique} provides a mechanism to produce them for any invariant admissible measure.

\begin{prop}\label{balancedlift}
Let  $f $ be a partially hyperbolic system. There is a neighbourhood $\mathcal{U}$ of $f$ in $\PH$ and $B>1$ such that, given any $g \in \mathcal{U}$ and every $\mu\in \AM_g(M)$, there exists a $B$-balanced lift for $\mu$, supported on $\A(a, C)$.
\end{prop}

\begin{proof}
 Let $\mu$ be an $f$-invariant admissible measure and pick some lift $\umu_0$ of $\mu$. 
By Proposition \ref{likePeSi}, we may suppose that $\umu_0$ is supported on $\A(a, C)$ for some large $C$. 
Since $\mu$ is invariant, $\umu_n =\Xi_{(f^n, a)} \umu_0$ is also a lift of $\mu$ for every $n \geq 0$. 
We will show that any accumulation point $\umu$ of $\umu_n$ is a $B$-balanced lift of $\mu$, and that $B$ can be chosen uniformly in a neighbourhood of $f$.

We define a map
\begin{eqnarray*}
 \Theta: \A &\rightarrow& \M(G^u M) \\
(\Gamma, \phi) &\mapsto& \int_{\Gamma} \delta_{T_z \Gamma} \ d(\Gamma, \phi)(z),
\end{eqnarray*}
giving the distribution of simple admissible measures in the Grassmannian bundle. For every $n \geq 0$, let 
$\tilde{\mu}_n = (Df^n)_* \int \Theta \  d\umu_0$. Denoting by $p$ the canonical projection $G^u M \rightarrow M$, sending $(x, h)$ into $x$, we certainly have $p_*\tilde{\mu}_n = f_*^n \mu = \mu$ for every $n\geq 0$. The statement that
$\Pi_* \umu \overset{B}{\sim} \mathsf{g}_* \mu^{\ominus}$ results from two claims:
\begin{enumerate}
 \item $\lim_{n\rightarrow \infty} \tilde{\mu}_n = \mathsf{g}_* \mu^{\ominus}$
\item $\Pi_* \umu_n \overset{B}{\sim} \tilde{\mu}_n$ for some $B\geq 1$ and every $n \geq 0$.
\end{enumerate}

To prove the first claim, pick an open set $\mathcal{O} \subset G^u M$ arbitrarily. We need to prove that 
\[\liminf_{n \rightarrow \infty} \tilde{\mu}_n (\mathcal{O}) \geq \mathsf{g}_* \mu^{\ominus}(\mathcal{O}).\] Let $A_i$ be the set of points $(x, h) \in \mathcal{O}$ such that, given any $y\in f^{-i}(x)$, if $(x, h) \in Df^i(S_y^u)$, then $Df^i(S_y^u) \subset \mathcal{O}$. The $A_i$ form an increasing sequence of open sets. The domination property (\ref{domination}) implies that $Df$ acts as a uniform contraction on $S^u$. Consequently, given any $(x, h) \in \mathcal{O}$, such an $i$ exists: 
\[
 \bigcup_{i\geq0} A_i \supset \mathcal{O} \cap S^u.
\]
Clearly the support of $\mathsf{g}_*\mu^{\ominus}$ is contained in $S^u$. Thus, given any $\epsilon >0$, we may choose some large $j$ so that $\mathsf{g}_*\mu^{\ominus}(A_j) > \mathsf{g}_*\mu^{\ominus}(\mathcal{O})-\epsilon$. By construction of $A_j$, $Df^{-j}(A_j)$ is an open set with the special property that 
\[Df^{-j}(A_j) \supset \bigcup_{x \in p(Df^{-j}(A_j))} S_x^u.\]
In particular, $\tilde{\mu}_n (Df^{-j} A_j) = \mu(p(Df^{-j} A_j)) = \mathsf{g}_*\mu^{\ominus}(Df^{-j}(A_j))$ for every $n \geq 0$. It follows from the commuting property $\mathsf{g}\circ \hat{f} = Df \circ \mathsf{g}$ that $\mathsf{g}_* \mu^{\ominus}$ is $Df$-invariant. We can therefore estimate
\begin{eqnarray*}
\liminf_{n \rightarrow \infty} \tilde{\mu}_n (A_j)
&=& \liminf_{n \rightarrow \infty} \tilde{\mu}_n (Df^{-j} A_j) \\
&=& \mu(p(Df^{-j}A_j)) 
= \mathsf{g}_* \mu^{\ominus}(Df^{-j}(A_j)) \\
&=& \mathsf{g}_*\mu^{\ominus}(A_j) 
\geq \mathsf{g}_*\mu^{\ominus}(\mathcal{O})-\epsilon.
\end{eqnarray*}
As $\epsilon$ may be taken arbitrarily small, we have indeed proved that 
$\tilde{\mu}_n \rightarrow \mathsf{g}_* \mu^{\ominus}$.

In order to prove the second claim, note that 
\begin{itemize}
\item $\displaystyle \tmu_n = \int \left( 
                        \int \delta_{Df^n(x) T_x \Gamma}\  d(\Gamma, \phi)(x) 
                                                              \right)d\umu_0(\Gamma,\phi)$,
\item $\displaystyle \Pi_* \umu_n = \int \left(
                          \int \delta_{Df^n(x) T_x \Gamma}\  d(\Gamma, \rho_{(\Gamma, \phi, f^n, a)})(x)
                                                  \right) d\umu_0(\Gamma, \phi)$. 
\end{itemize}
Consequently, the second claim follows if we can find $B$ such that
\[
                B^{-1} \phi(x) \leq \rho_{(\Gamma, \phi, f^n, a)}(x) \leq B \phi(x)
							\]
for every $(\Gamma, \phi) \in \A$ and $x \in \Gamma$.
We shall prove the second inequality. The first one is analogous.

Fix $n \geq 0$ and $(\Gamma, \phi) \in \A$ arbitrarily. 
By definition,
\begin{eqnarray*}
   \rho(x) = \rho_{(\Gamma, \phi, f^n, a)}(x)
         &=& \int_{W_x} \frac{\phi(y)}{(\Gamma, 1)(V_y)} d(\Gamma, 1)(y) \\
          & \leq &  \frac{ \sup \{ \phi(y) : y \in W_x \} (\Gamma, 1)(W_x)}
	{ \inf \{(\Gamma, 1)(V_y): y \in W_x\} }  \\
            &\leq& C^2 \phi(x) \frac{(\Gamma, 1)(W_x)}
{\inf \{(\Gamma,1)(V_y): y \in W_x\}}.
\end{eqnarray*}
As already observed in the proof of Lemma \ref{Lip1}, $W_x \subset B_{3 R_a(x)/2}^{\Gamma,n}(x)$. Hence $(\Gamma, 1)(W_x) \leq (\frac{3}{2})^u C_1 \vol(\mathbb{B}^u)R_a(x)^u$, where $C_1$ is the constant described in the proof of Lemma \ref{Lip1}.
As $R_a(x)$ is $\frac{1}{2}$-Lipschitz with respect to the metric $d^{\Gamma,n}$, we find that $\inf_{y \in W_x} R_a(y) \geq R_a(x) - \frac{1}{2}\frac{3}{2}R_a(x) = \frac{1}{4}R_a(x)$ wherefore, again by Sublemma \ref{estimate}, 
\[
\inf_{y\in W_x}(\Gamma, 1)(V_y) \geq \inf_{y\in W_x} B_{R_a(x)/8}(y) \geq 8^{-u} C_1^{-1} \vol(\mathbb{B}^u)R_a(x)^u.
\]
The proof follows by taking $B = 12^u C^2 C_1^2$.
\end{proof}

\subsection{Statistical stability}
Before we indulge in the proof of parts 2 and 3 of Theorem B, let us make some important preliminary observations. First note that a carrier 
$(x, r, h, \phi)$ is quite well described by its three first coordinates, provided $r$ is sufficiently small. Indeed, the nonlinear displacement $\phi$ is $C^1$ close to $\exp_x\vert_h$ in a uniform fashion, due to the Lipschitz condition. Now let us fix a choice of metric $d_G$ in $G^u M$. Then, provided $a$ is small enough, there exists $\delta>0$ with the following property: Let $g\in \mathcal{U}$, where $\mathcal{U}$ is as in Lemma \ref{size}, and suppose that $\Gamma_1$ is some carrier such that $(\Gamma_1,1) \in \mathcal{L}(g)\cap \A(a)$. Then, given any carrier $\Gamma_2$ with 
$d_G (\Pi(\Gamma_1), \Pi(\Gamma_2))<\delta$, we have  $(\Gamma_1, 1)(D(\Gamma_1, \Gamma_2))>0$.

For better appreciation of the proof, let us first go through it in loose terms. If a system $g$ is close to $f$,  then every physical measure $\nu$ of $g$ is close to the finite dimensional simplex whose vertices are the physical measures of $f$, say $\mu_1, \ldots, \mu_N$. We know that there is a fairly large portion of large stable manifolds through many $\nu$-generic carriers (Corollary \ref{manylargemanifolds}). These are located near $\supp \mu_1 \cup \ldots \cup \supp \mu_N$, although we do not know precisely where. A priori they could all be cuddled up near one of the sets $\supp \mu_i$ for which $\nu$ tends to give positive weight. The idea is to prove that this hinders any other physical measure $\nu'$ of $g$ to give positive weight to $\mu_i$. We do that by using proposition \ref{balancedlift} about existence of balanced lifts. It means that if $\nu'$ were indeed to give positive weight to (a neighbourhood of) $\mu_i$, then it would have generic carriers close to some $\nu$-generic one, possessing a big portion of large stable manifolds. But then, as remarked above, absolute continuity of the stable foliation would force these carriers to be generic for the same measure, which is a contradiction. Hence every physical measure of $g$ 'occupies' one physical measure of $f$ in a one-to-one manner, so the number of physical measures of $g$ is at most that of $f$. The only way it could be equal is if every physical measure $\nu_i$ of $g$ occupies precisely one of the $\mu_i$, and gives no weight to the others.

\begin{proof}[Proof of Theorem \ref{thm2} part 2 and 3]

Let $\mu_1, \ldots, \mu_N$ denote the physical measures of $f$. As remarked upon in Section \ref{ILBL}, to each $\mu_i$, there is a unique inverse limit $\mu_i^{\ominus}$, invariant under $\hat{f}$. For every $1\leq i \leq N$, we cover $\supp \mu_i^{\ominus}$ by a finite number of balls $B_{i j}:= B_{\delta/2}(x_j), \quad 1\leq j \leq m_i$. Thus $B_i := \bigcup_{j=1}^{m_i}B_{i j}$ 
is a neighbourhood of $\supp \mu_i^{\ominus}$.

Choose a $C^2$ neighbourhood $\mathcal{U}$ of $f$ satisfying the conclusions of both Proposition \ref{balancedlift} and Lemma \ref{size}. Moreover, $\mathcal{U}$ should be small enough so that if $g \in \mathcal{U}$ and $\nu_1 \ldots \nu_{N'}$ are the physical measures of $g$, then for every $1\leq l \leq N'$,
\begin{enumerate}
 \item there exists $1\leq i \leq N$ such that 
$\nu_l^{\ominus}(B_{i j})>0 \quad \forall 1\leq j \leq m_i$,
\item $\nu_l^{\ominus}(\bigcup_{i=1}^N B_i)>1-\frac{\theta}{2C}$.
\end{enumerate}
(The map $\mu \mapsto \mu^{\ominus}$ is linear continuous.) The constant $C$ is large enough so that each $\nu_l$ has a lift supported on $\A(a, C)$.

Our aim is to prove that $N' \leq N$. Choose $a$ small enough for the remarks in the beginning of this section to apply. By Proposition \ref{balancedlift}, there exist $B$-balanced lifts $\unu_1, \ldots, \unu_{N'}$ of $\nu_1, \ldots, \nu_N$, all supported on $\A(a, C)$. It follows from the second item above that, given any $\nu_l$ there is some ball $B_{i j}$ with $\nu_l^{\ominus}(B_{i j}\cap \Pi(\mathcal{L}(g))>0$. Since $\unu_l$ is balanced, and $\unu_l$-almost every carrier is generic for $\nu_l$ we infer the existence of a $\nu_l$-generic carrier $\Gamma_{i j}^l$ such that $\Pi(\Gamma_{i j}^l) \in B_{i j}$ for some $i\in \{1, \ldots, N\}$ and $j\in \{1,\ldots, m_i\}$. 

We claim that if $k\neq l$, then it is impossible to have 
$\nu_k^{\ominus}(B_{i j}) >0 $ for every $1\leq j\leq m_i$. Otherwise, there would be some $\nu_k$-generic carrier $\Gamma_{i j}^k$ with $\Pi(\Gamma_{i j}^k) \in B_{i j}$. That would imply that
$d_G(\Gamma_{i j}^l, \Gamma_{i j}^k) <\delta$, and since $(\Gamma_{i j}^l, 1) \in \mathcal{L}(g)$ we conclude 
$(\Gamma_{i j}^l, 1)(\mathcal{D}(h_{\Gamma_{i j}^l, \Gamma_{i j}^k}))>0$ which is absurd.

We have shown that each $\nu_l$ can be associated to some $\mu_i$ in a one-to-one manner, namely by asking that $\nu_l^{\ominus}(B_{i j}\cap \Pi(\mathcal{L}(g))$ be positive for some $j= 1, \ldots, m_i$. Consequently $N'\leq N$ and the second part of Theorem B is proved. 

Suppose now that $N = N'$. Since $S$ is closed, each $\nu_l$ must be close to some convex combination $\alpha_1 \mu_1 + \ldots + \alpha_N \mu_N$. Hence $\nu_l^{\ominus}$ must be close to $\alpha_1 \mu_1^{\ominus} + \ldots + \alpha_N \mu_N^{\ominus}$. We have already seen that $\nu_l^{\ominus}(B_{i j}\cap \Pi(\mathcal{L}(g))$ would imply $\alpha_k = 0$ for every $k\neq i$. Hence every $\nu_l$ is near some $\mu_i$ and the third part of Theorem B is proved.
\end{proof}

\section{Stochastic stability}

We have seen in section \ref{ergodic} that for a system $f\in \PH$, any invariant admissible measure is a convex combination of ergodic admissible measures. If, furthermore, $f$ is has mostly contracting central direction, then every ergodic admissible measure is a physical measure. Hence, in order to prove that maps in $MC$ are stochastically stable, it suffices to prove that every zero noise limit is admissible.

\begin{prop}\label{stationarydist}
Let $f\in \PH$ and suppose that $\{\nu_{\epsilon}\}_{\epsilon}$ is a local absolutely continuous perturbation scheme. Then every zero noise limit is admissible.
\end{prop}

Let $\Omega = \Dl(M)^{\mathbb{Z}_+}$ and write $\nue^{\mathbb{Z}_+}$ for the Bernoulli measure on $\Omega$. Given $\ff = (f_0, f_1, \ldots)\in \Omega$, we shall write 
$\ff_n = f_{n-1}\circ \ldots \circ f_1 \circ f_0$. Replacing $f^n$ with $\ff_n$ in Proposition \ref{curvature} (and its proof), we obtain

\begin{prop}\label{randomcurvature}
The family of admissible manifolds is $\ff_n$-invariant for $\nu_{\epsilon}^{\mathbb{Z}_+}$-almost every $\ff$ and every $n \geq 0$, provided that $\epsilon$ is small enough.
\end{prop}

Proposition \ref{randomcurvature} allows us to mimic the construction in section \ref{technique}. Indeed, given some simple admissible measure 
$(\Gamma, \phi)$, we consider the map
\begin{eqnarray*}
 \xi_{(\Gamma, \phi, \ff_n, a)}: \Gamma &\rightarrow& \A \\
x &\mapsto& (\Gamma_x, \phi_x)
\end{eqnarray*}
defined in section \ref{technique},
along with the density $\rho_{(\Gamma, \phi, \ff_n, a)}$ on $\Gamma$, such that 
\[(\xi_{(\Gamma, \phi, \ff_n, a)})_*(\Gamma, \rho_{(\Gamma, \phi, \ff_n, a)})\] 
is a lift of $(\ff_n)_* (\Gamma, \phi)$. Similarily, if $\umu$ lifts $\mu$, then $\Xi_{(\ff_n, a)} \umu$ lifts $(\ff_n)_*\mu$.

Recall that $\T^n \mu = \int_{\Omega} (\ff_n)_* \mu \ d\nue^{\mathbb{Z}_+}$. Hence we we may define a random operator
\begin{eqnarray*}
\Xi_{(\nu_{\epsilon}, n, a)}^{\text{rand}} :\M(\A) &\rightarrow& \M(\A) \\
\umu &\mapsto& \int \Xi_{(\ff_n, a)} \umu d\nue^{\mathbb{Z}_+}. 
\end{eqnarray*}
with the delightful property that if $\umu$ lifts $\mu$, then $\Xi_{(\nu_{\epsilon}, n, a)}^{\text{rand}} \umu$ lifts $\T^n \mu$.  We infer that $\AM$ is invariant under $\T$.

\begin{proof}[Proof of Proposition \ref{stationarydist}]
 Suppose $\mu_{\epsilon}$ is an invariant distribution under $\T$ 
and let $E\subset M$ be any Borel set of zero Lebesgue measure. Then, from (\ref{abscont}), we have
\[\mue(E) = \T \mue (E) = \int \T \delta_x (E) \ d\mue(x) = 0.\]
Hence $\mue$ is absolutely continuous with respect to Lebesgue, and it follows that it can be strongly approximated by an admissible measure. Thus given $\delta>0$ arbitrarily, we may pick some admissible $\mu$ with lift $\umu$, satisfying $\umu(\A_1) = 1$ (recall Proposition \ref{suppA1}),  and such that $\|\mue-\mu\|_s \leq \delta$.

We extend $\T$ to $C^0(M, \mathbb{R})^*$ by requiring
\[ \T \mu (\varphi) = \int_{\Dl(M)} \mu(\varphi \circ f) \ d\nue(f) \quad \forall \varphi \in C^0(M, \mathbb{R}). \]
Given any $\varphi \in C^0(M)$ with $\|\varphi \|_{C^0} \leq 1$, we have 
\[\T \mu(\varphi) \leq \int_{\Dl(M)} \| \mu \|_s \ d\nue(f) = \| \mu \|_s , \]
so that 
\[\|\T \mu \|_s = \sup_{\substack{ \varphi \in C^0(M, \mathbb{R}) \\ \| \varphi \|_{C^0} \leq 1 }}
 \T \mu(\varphi) \leq \| \mu \|_s. 
\]
 In other words, $\T$ acts as a contraction on $C^0(M, \mathbb{R})^*$.
Hence
\[\|\mue - \frac{1}{n} \sum_{k=0}^{n-1}\T^k \mu \|_s
= \|\frac{1}{n} \sum_{k=0}^{n-1} \T^k (\mue-\mu)\|_s  \leq \delta \]
for every $n\geq0$.
Since $\frac{1}{n} \sum_{k=0}^{n-1}\T^k \mu_0$ accumulates on some $\AM(a, C)$, it follows that 
\[\inf_{\mu \in \AM(a, C)} \|\mue - \mu \|_s \leq \delta\] for every $\delta>0$. Therefore, by compactness, $\AM(a, C)$ must contain $\mue$.

By now it should now be evident that every zero noise limit is admissible. Indeed, since $\AM(a, C)$ is a compact space, it contains any accumulation points of stationary distributions $\mue \in \AM(a, C)$. 
\end{proof}

\end{document}